\documentclass[12pt]{amsart}
\usepackage{hyperref}
\usepackage{amsaddr}
\usepackage[margin=1in]{geometry}
\usepackage{wrapfig, caption}
\usepackage[utf8]{inputenc}
\usepackage{amsmath, graphicx}
\usepackage{color, bbold}
\usepackage{subcaption}
\usepackage{amsthm}
\usepackage{amssymb}
\usepackage{amscd}
\usepackage{mathtools}
\usepackage{tikz}
\usepackage{comment}
\usepackage{cleveref}
\usepackage{cancel}
\usepackage{graphicx}
\usepackage{diagbox}
\usepackage{microtype}
\usetikzlibrary{shapes.geometric}
\usepackage[shortlabels]{enumitem} 
\usepackage[numbers]{natbib}
\usepackage{float}

\usetikzlibrary{patterns}
\tikzset{
  norm/.style     = {shape=circle, draw},
  blue/.style     = {shape=circle, draw, fill=blue!25},
  high/.style     = {shape=circle, draw, color=red},
  bluehigh/.style = {shape=circle, draw, color=red, fill=blue!25},
  red/.style      = {shape=circle, draw, fill=red!25},
  both/.style     = {shape=circle, draw, fill=violet!35},
  root/.style     = {node, bottom color=red!30},
  env/.style      = {treenode, font=\ttfamily\normalsize},
  dummy/.style    = {circle}
}
\tikzstyle{standard}=[circle, draw=black, fill=white, very thick, minimum size=7mm]
\tikzstyle{standard2}=[circle, draw=black, fill=white, very thick]
\tikzstyle{blue2}=[circle, draw=black, fill=blue!25, very thick]
\tikzstyle{small}=[circle, draw=black, fill=black, very thick, minimum size=4mm]
\tikzstyle{special}=[circle, draw=red!60, fill=red!5, very thick, minimum size=5mm]
\newtheorem{theorem}{Theorem}[section]
\newtheorem{lemma}[theorem]{Lemma}

\theoremstyle{definition}
\newtheorem{df}[theorem]{Definition}

\newtheorem{conj}[theorem]{Conjecture}

\newtheorem{example}[theorem]{Example}

\DeclareMathOperator{\lk}{lk}
\DeclareMathOperator{\st}{st}
\DeclareMathOperator{\del}{del}

\newcommand{\Cl}{\mathsf{Cl}}

\title{Homotopy Type of Total Cut Complexes of Squared Cycle Graphs}
\author[Y.F. Shen]{Yufeng Shen}
\address{Xi’an Jiaotong University, Xi’an, Shaanxi 710049, China, yufeng\_shen@stu.xjtu.edu.cn}
\author[Z.Y. Song]{Zhiyu Song}
\address{Nankai University, Tianjin 300071, China, 2210655@mail.nankai.edu.cn}
\author[F.L. Yu]{Fenglin Yu}
\address{Peking University, Beijing 100871, China, fenglin@stu.pku.edu.cn}
\author[L.W. Zhou]{Leopold Wuhan Zhou}
\address{Peking University, Beijing 100871, China, wuhanzhou@stu.pku.edu.cn}
\author[J.Q. Zhuang]{Jingqi Zhuang}
\address{Fudan University, Shanghai 200433, China, 22300680047@m.fudan.edu.cn}

\date{\today}
\begin{document}
\begin{abstract}
    In this paper, we investigate the homotopy type and combinatorial properties of \emph{total cut complexes} of squared cycle graphs. The total cut complexes are a new type of graphical complexes introduced by Bayer et al in \cite{Bayer_2024} to extend Fröberg's theorem. In \cite{Bayer_2024_02}, the authors made a conjecture on the homotopy type of total cut complexes of squared cycle graphs for $k\ge3$. We proved this conjecture in the case $k=3$. For general $k\geq 3$, we confirmed the cases when $n=3k+1$ and $3k+2$.
\end{abstract}
\maketitle
\section{Introduction}\label{sec1:intro}
In \cite{Bayer_2024}, \cite{Bayer_2025} and \cite{Bayer_2024_02}, Bayer et al. introduced two new families of graph complexes called \emph{cut complexes} and \emph{total cut complexes}. Given a graph\footnote{All the graphs in this article are assumed to be finite and simple, that is, without loops
and multiple edges. } $G=(V,E)$, the \textit{total cut complex} $\Delta_k^t(G)$ is the simplicial complex whose facets are the complements of independent sets of size $k$ in graph $G$. One of the main motivations
behind their work was a famous theorem of Ralf Fröberg \cite{Froberg_1990} connecting commutative algebra and graph theory through topology. 
\begin{theorem} [{\cite[Theorem 1]{Froberg_1990}, \cite{Eagon_1998}}] \label{thm:Fr}

	A Stanley–Reisner ideal $I_{\Delta}$ generated by quadratic square-free monomials has a $2$-linear resolution if and only if $\Delta$ is the clique complex $\Cl(G)$ of a chordal graph $G$.
\end{theorem}

Another reason why we focus on the total cut complexes of squared cycle graphs is the following. In \cite{Bayer_2024_02}, Bayer et al. identified the homotopy type of the $k$-total cut complex of $C_n$:
\begin{theorem}
\label{thm:total-cut-cycle}
(\cite{Bayer_2024_02}) For $n<2k$, $\Delta_k^t(C_n)$ is the void complex and hence shellable.
For $n\ge 2k\ge 4$,
$\Delta_k^t(C_n)$ is homotopy equivalent to a single sphere in dimension $n-2k$.
\end{theorem}

The following theorem \ref{Motivation-shen}, which is first proved by Mark Denker and Lei Xue, indicates that the total cut complex of the cycle graph and the neighborhood complex of the stable Kneser graph $SG(n,k)$ (\cite{Bjorner_2003}) have the same homotopy type.
\begin{theorem}\label{Motivation-shen}
    Let $SG(n,k)$ be the stable Kneser graph for $n\ge1,k\ge1$. The $k$-total cut complex of the $n$-cycle is a \textbf{nerve complex} of a good cover of the neighborhood complex of $SG(n,k)$, namely $\mathcal{N}(SG(n,k))$. Hence, by the nerve lemma we have $$\Delta_k^t(C_{n})\simeq\mathcal{N}(SG(n,k))$$
\end{theorem}

This surprising connection suggests that the topology of total cut complexes may have some relationships with the neighborhood complexes of some graphs which are difficult to handle . Hence, it is another motivation for us to study the total cut complexes.

Specifically, we are looking forward to the topology structure of total cut complexes of more classes of graphs. The following conjecture was made in \cite{Bayer_2024_02}.

\begin{conj}[\cite{Bayer_2024_02},Conjecture 5.1]\label{conj-Wn}
    Let $W_n$ be squared cycle graph with $n$ vertices. The $(n-k-1)$-dimensional complex $\Delta_k^t(W_n)$ is the void complex if $n\le 3k-1$. Otherwise, it is homotopy equivalent to a single sphere in dimension 
\[\begin{cases} 2i+1, & n=3k+i, \ 0\le i\le k-1,\\
               2k+i, & n=4k+i, \ i\ge 0,
\end{cases} \]
\end{conj}

In this paper, we introduce a new way to describe the faces of the total cut complexes of squared cycle graphs $\Delta_k^t(W_n)$ and partially prove Conjecture \ref{conj-Wn} for $k=3$. For general $k\geq 3$, we prove specific special cases as Theorem \ref{zjq-zwh-k=3} and Theorem \ref{zjq-syf-n=3k+1}.

\begin{theorem}\label{zjq-zwh-k=3}
    \begin{equation}
        \Delta_{3}^t(W_{n})\simeq \begin{cases}
            \mathbb{S}^{2n-17}, \quad n = 9, 10\\
            \mathbb{S}^{n-6}, \quad n \geq 11\\
        \end{cases}
    \end{equation}
\end{theorem}
\begin{theorem}\label{zjq-syf-n=3k+1}
     For $k\geq 3$, we have:
         \begin{equation}
        \Delta_k^t(W_{3k+1}) \simeq \mathbb{S}^3
    \end{equation}
    \begin{equation}
        \Delta_k^t(W_{3k+2}) \simeq \mathbb{S}^5
    \end{equation}
\end{theorem}

The structure of this paper is as follows: 

\begin{itemize}
    \item Section \ref{sec2:preli} introduces some basic definitions and results related to total cut complexes and graphs in general.
    \item Section \ref{sec3:main} focuses on computing the homotopy type of $\Delta^t_k (W_n)$, the $k$-total cut complex of the squared cycles in the case when $k=3$ ( Theorem \ref{zjq-zwh-k=3}).
    \item Section \ref{mainsec2} introduces a new way to describe the faces of the total cut complexes of squared cycle graphs $\Delta_k^t(W_n)$ and proved Theorem \ref{zjq-syf-n=3k+1}.
    \item Section \ref{sec5:conclu} proposed several conjectures and questions about the homotopy types of the complexes $\Delta_k^t(W_n)$ based on our SageMath data.
\end{itemize}

\section{Preliminaries}\label{sec2:preli}
In this section, we give some basic definitions and useful results related to total cut complexes and discrete Morse matching. General references for simplicial complexes and topology are \cite{TopMeth}, \cite{Jonsson_2007} and \cite{Hatcher_2002}. 
\subsection{Simplicial Complexes}
A \textit{simplicial complex} $\Delta$ on a set $A$ is a collection of subsets of $A$ such that 
\[\sigma\in \Delta \text{ and } \tau\subseteq \sigma \Rightarrow \tau \in \Delta.\] 
The elements of $\Delta$ are its \emph{faces} or \emph{simplices}. 
If the collection of subsets is empty, we call $\Delta$ the void complex. 
The \emph{dimension} of a face $\sigma$, $\dim(\sigma)$, is one less than its cardinality. A {\em facet} of a simplicial complex is a maximal face. The dimension of a simplicial complex is the maximum of the dimensions of its simplices. The complexes we will be considering are {\em pure}, meaning that all facets are of the same dimension.

We will be using the fowlling operations.
\begin{df}\cite{Jonsson_2007}\cite{Kozlov_2008}\label{operations}
    Let $\Delta$ be a simplicial complex and $\sigma$ a face of $\Delta$.
    \begin{itemize}
       \item The \emph{link} of $\sigma$ in $\Delta$ is 
    $$\lk_{\Delta} \sigma= \{\tau \in \Delta \mid \text{$\sigma\cap \tau 
      = \varnothing$, and $\sigma\cup \tau \in \Delta$}\}.$$
\item The (closed) \emph{star} of $\sigma$ in $\Delta$ is 
     $$\st_{\Delta} \sigma= \{\tau \in \Delta \mid  
        \sigma \cup \tau \in \Delta\} .$$
\item The \emph{deletion} of $\sigma$ in $\Delta$ is 
      \[\del_{\Delta} \sigma = \{\tau \in \Delta \mid \sigma \not\subseteq
       \tau\}.\]

    \end{itemize}
\end{df}
\subsection{Discrete Morse Theory}

We present the essential definitions and key tools from discrete Morse Theory that are utilized in this paper. For more details, refer to \cite{FORMAN199890}.
\begin{df}\cite{FORMAN199890}
A partial matching in a poset $(P,\text{ \textless})$ is a subset $M \subseteq P \times P$, such that:
\begin{enumerate}
    \item[(i)] If $(a, b) \in M$, then $b \succ a$, meaning $a < b$ and there is no $c \in P$ satisfying $a < c < b$; and
    
    \item[(ii)] Every $a \in P$ is contained in at most one pair in $M$.
\end{enumerate}

If $(a,b) \in M $, we denote $ a $ by $ d(b) $ and $ b$ by $ u(a) $. A partial matching on $P$ is called \emph{acyclic} if there do not exist distinct $a_{i}\in P$, $1\leq i \leq  m$, $ m \geq 2 $ such that 
\begin{equation} \label{eq3.1}
 a_1 \ \prec \ u(a_1) \ \succ \ a_2 \ \prec \ u(a_2) \ \succ \ \ldots \ \prec \ u(a_m) \ \succ \ a_1 \  \ {\text is\ \ a\ \ cycle}.
\end{equation}

After applying an acyclic partial matching $\mathcal{M}$ on $P$, those elements which not belong to the matching are called  \emph{unmatched} or \emph{critical cells} with respect to the matching $\mathcal{M}$.  
\end{df}

We use the following type of matchings as the main tool in computing the homotopy types. 

\begin{df}\cite{Jonsson_2007,Singh_2021}\label{def:eltmatch}
Given $\Delta$ a simplicial complex and $i \in [n]$ a vertex. The \emph{element matching} on $\Delta$ using $i$ is the matching 
\[ \mathcal{M}_{i} : = \{(\sigma, \sigma\sqcup\{i\}) \mid \sigma\sqcup\{i\}\in \Delta, i \notin \sigma\}.\]
\end{df}
The following result tells us that a sequence of element matchings is always acyclic. 

\begin{theorem}\cite{Jonsson_2007}\label{element-acyclic1}
Let $\Delta$ be a simplicial complex and $\{x_1, x_2,\cdots, x_n\}$ be a subset of the vertex set of  $\Delta$. Let $\Delta_0= \Delta$ and for $i\in [n]$ define 
\begin{align*}  M(x_i)&\coloneqq\{(\sigma, \sigma\cup\{x_i\}) \mid x_i\not\in \sigma, \textrm{ and } \sigma, \sigma\cup\{x_i\}\in \Delta_{i-1}\},\\
N(x_i)&\coloneqq \{\sigma\in \Delta_{i-1} \mid \sigma\in\eta \textrm{ for some } \eta\in M(x_i)\}, \textrm{ and }\\
\Delta_i&\coloneqq\Delta_{i-1}\setminus N(x_i).
\end{align*}

Then $\bigsqcup_{i=1}^n M(x_i)$ is an acyclic matching on  $\Delta$.
\end{theorem}
\begin{theorem}[\cite{Kozlov_2008},Thm.11.13]
Let $\Delta$ be a simplicial complex and $\mathcal{M} $ be an acyclic matching on the face poset of $\Delta$. Denote by $c_i$ the number of critical $i$-dimensional cells of $\Delta$ with respect to the matching $\mathcal{M}.$  Then $\Delta$ is homotopy equivalent to a cell complex $\Delta_c$ that has exactly $c_i$ cells in each dimension $i\ge 0$, with the addition of a single $0$-dimensional cell, when the empty set is also paired.

In particular, if an acyclic matching has critical cells in only one fixed dimension $i\ge 0,$ then $\Delta$ is homotopy equivalent to a wedge of $i$-dimensional spheres.
\end{theorem}

\section{3-total cut complex of squred cycles}\label{sec3:main}
In this section, we prove the first main result (\Cref{zjq-zwh-k=3}) of this article. Throughout the section, we use the discrete Morse theory as our main tool. We aim to construct an efficient element matching. First, we recall the following definitons.
\begin{df}
    The \textit{squared cycle graph} $W_n$ is the graph with vertex set $[n]$, and edge set $\{\{i, i+1 \pmod n\}, \{i, i+2 \pmod n \}\}, i = 1, 2, \cdots, n$.
\end{df}
We label the vertices of $W_n$ by $1, 2, \cdots , n$ in cyclic order. Note if $n\leq 5$, $W_n$ is the complete graph $K_n$.

To study the $k$-total cut complex of $W_n$, we need to determine when a subset $A$ of vertices forms a face of $\Delta_k^t(W_n)$. At this point we have 
\begin{lemma}
    \label{lem:3.1}
    For $n \ge 3k+1$, any subset $A$ of $3k-2$ vertices must contain a $k$-independent set in $W_n$. Therefore, every subset of $[n]$ with cardinality at most $n-3k+2$ is a face of $\Delta_k^t(W_n)$.
\end{lemma}
\begin{proof}
    Let $A$ be a subset of $3k-2$ vertices. For $n > 3k + 1$, we can remove $n - (3k + 1)$ vertices that are not in $A$ from $W_n$ without increasing the independent number of $A$. Therefore, it suffices to show that any subset of $[3k + 1]$ with cardinality $3k - 2$ contains a $k$-independent set in $W_n$.

    If $[3k + 1] \setminus A$ is a set of three consecutive vertices, then $A$ is a squared path and contains a $k$-independent set. 

    If $[3k + 1] \setminus A$ contains two adjacent vertices and one vertex not adjacent to either, then $A$ is divided into two squared paths of cardinality $c_1$ and $c_2$, where $c_1 + c_2 = 3k - 2$. If $c_1 \equiv c_2 \equiv 2 \pmod 3$, then $A$ contains a $\frac{c_1+1}{3} + \frac{c_2+1}{3} = k$-independent set. If $c_1 \equiv 0 \pmod 3$ or $c_2 \equiv 0 \pmod 3$, then $A$ contains a $\frac{c_1}{3} + \frac{c_2+2}{3} = k$-independent set.

    If $[3k + 1] \setminus A$ are three nonadjacent vertices, then $A$ is divided into three squared paths of cardinality $c_1, c_2, c_3$, where $c_1 + c_2 + c_3 = 3k - 2$. Then there are three cases (considering the symmetry of the three parts):
    \begin{itemize}
        \item If $c_1 \equiv c_2 \equiv 0 \pmod 3$, then $A$ contains a $\frac{c_1}{3} + \frac{c_2}{3} + \frac{c_3 + 2}{3} = k$-independent set.
        \item If $c_1 \equiv c_2 \equiv 1 \pmod 3$, then $A$ contains a $\frac{c_1 + 2}{3} + \frac{c_2 - 1}{3} + \frac{c_3 + 1}{3} = k$-independent set.
        \item If $c_1 \equiv c_2 \equiv 2 \pmod 3$, then $A$ contains a $\frac{c_1 + 1}{3} + \frac{c_2 + 1}{3} + \frac{c_3}{3} = k$-independent set.
    \end{itemize}
    Therefore, we have shown that any subset of $[3k + 1]$ with cardinality $3k - 2$ contains a $k$-independent set, and thus every subset of $[n]$ with cardinality at most $n - 3k + 2$ is a face of $\Delta_k^t(W_n)$.
\end{proof}
With Lemma \ref{lem:3.1} and discrete Morse theory, we divide the faces of $\Delta_3^n(W_n)$ into several categories, thus leading to the following theorem. 
\begin{theorem}\label{Zjq-k=3}
    \begin{equation}
        \Delta_{3}^t(W_{n})\simeq \begin{cases}
            \mathbb{S}^{2n-17}, \quad n = 9, 10\\
            \mathbb{S}^{n-6}, \quad n \geq 11\\
        \end{cases}
    \end{equation}
\end{theorem}
\begin{proof}
    If $n = 9$, $\Delta_{3}^t(W_{n})$ has facets $\{124578, 134679, 235689\}$ and it is easy to see this is homotopy equivalent to $\mathbb{S}^{1}$.
    
    If $n = 10$, we will show in \cref{zjq-syf-n=3k+1} that in this case $\Delta_{3}^t(W_{10}) \simeq \mathbb{S}^{3}$.

    For $n \geq 11$, we exhibit a Morse matching. Recall that each facet of $\Delta_{3}^t(W_{n})$ has $n - 3$ vertices, so $\dim \left(\Delta_{3}^t(W_{n})\right) = n - 4$.

    Frist, we apply the element matching $\mathcal{M}_1$. A face $\sigma \in \Delta_{3}^t(W_{n})$ is unmatched if and only if $1 \notin \sigma$ and $\sigma \cup \{1\} \notin \Delta_{3}^t(W_{n})$. 
    By Lemma \ref{lem:3.1}, the unmatched faces can contain $n-3, n-4, n-5, n-6$ or $n-7$ vertices.
    We can divide the unmatched faces into the following cases:
    \begin{itemize}
        \item There are $\frac{1}{2}(n-7)(n-8)$ faces of cardinality $n-3$ which are unmatched. Each of them is of the following form.
        \[
        A_{i, j} = [n] \setminus \{1, i, j\}, \quad 4 \leq i < i+3 \leq j \leq n-2.
        \]
        
        \item There are four types of unmatched faces of cardinality $n-4$. Each type contains faces that are complements of the union of $\{1\}$ and a set of two adjacent vertices and a disjoint vertice. We denote them by:
        $$B_{i, j} = [n] \setminus \{1, i, i+1, j\}, \quad i = 3, 7 \leq j \leq n-2 \text{ and } 4 \leq i < i+3 \leq j \leq n-2$$
        $$C_{i, j} = [n] \setminus \{1, i, i+2, j\}, \quad i = 2, 3 i +5 \leq j \leq n-2 \text{ and } 4 \leq i \leq n-5, i+3 \leq j \leq n-2 $$
        $$D_{i, j} = [n] \setminus \{1, i, j, j+1\}, \quad 4 \leq i < i+3 \leq j \leq n-2$$
        $$E_{i, j} = [n] \setminus \{1, i, j, j+2\}, \quad 4 \leq i < i+3 \leq j \leq n-2.$$
        
        \item The unmatched faces of cardinality $n-5$ are of six types. Each type is characterized by either three adjacent vertices $k,k+1,k+2$ or two pairs of adjacent vertices $i,i+1,j,j+1$, as shown below:
        $$F_{i, j} = [n] \setminus \{1, i, i+1, j, j+1\}, \quad i = 3, 6 \leq j \leq n-2 \text{ and } 4 \leq i \leq n-5, i+2 \leq j \leq n-2$$
        $$G_{i, j} = [n] \setminus \{1, i, i+1, i+2, j\}, \quad i = 2, 7 \leq j \leq n-2 \text{ and } 3 \leq i < i+4 \leq j \leq n-2$$
        $$H_{i, j} = [n] \setminus \{1, i, j, j+1, j+2\}, \quad 4 \leq i \leq n-5, i+2 \leq j \leq n-2$$
        $$I_{i, j} = [n] \setminus \{1, i, i+2, j, j+1\}, \quad 2 \leq i < i+4 \leq j \leq n-2$$
        $$J_{i, j} = [n] \setminus \{1, i, i+1, j, j+2\}, \quad 3 \leq i < i+3 \leq j \leq n-2$$
        $$K_{i, j} = [n] \setminus \{1, i, i+2, j, j+2\}, \quad 2 \leq i < i+3 \leq j \leq n-2.$$
        
        \item Similarly, there are four types of unmatched faces of cardinality $n-6$:
        $$L_{i, j} = [n] \setminus \{1, i, i+1, i+2, j, j+1\}, \quad i = 2, 6 \leq j \leq n-2 \text{ and } 3 \leq i < i+3 \leq j \leq n-2$$
        $$M_{i, j} = [n] \setminus \{1, i, i+1, j, j+1, j+2\}, \quad 3 \leq i < i+2 \leq j \leq n-2$$
        $$N_{i, j} = [n] \setminus \{1, i, i+2, j, j+1, j+2\}, \quad 2 \leq i < i+3 \leq j \leq n-2$$
        $$O_{i, j} = [n] \setminus \{1, i, i+1, i+2, j, j+2\}, \quad 2 \leq i < i+3 \leq j \leq n-2.$$
        \item The unmatched faces of cardinality $n-7$ are complements of the union of $\{1\}$ and two disjoint set of three adjacent vertices $k,k+1,k+2$. There are $\frac{1}{2}(n-6)(n-5)$ faces of this type:
        $$P_{i, j} = [n] \setminus \{1, i, i+1, i+2, j, j+1, j+2\}, \quad 2 \leq i < i+3 \leq j \leq n-2.$$
    \end{itemize}
    We proceed with the following matchings:
    $$\mathcal{M}_2, \mathcal{M}_3, \cdots, \mathcal{M}_{n-6}$$
    where $\mathcal{M}_k$ denotes the element matching using vertex $k\in \{2,3,..,n-6\}$. For all $2\leq k \leq n-6$, the matching $\mathcal{M}_k$ results in the following pairings:
    
        $$(A_{k+2,j},C_{k,j}),\quad 2\leq k\leq n-7,\quad k+5\leq j\leq n-2$$
        $$(B_{k+2,k+5},K_{k,k+3}),\quad 2\leq k\leq n-7$$
        $$(B_{k+1,j},G_{k,j})\quad 2\leq k \leq n-7,\quad k+5\leq j \leq n-2$$
        $$(C_{k+2,k+5},I_{k,k+4}),\quad 2\leq k \leq n-7$$
        $$(C_{k+2,k+6},K_{k,k+4}),\quad 2\leq k\leq n-8$$
        $$(D_{k+2,j},I_{k,j}),\quad 2\leq k \leq n-7,\quad k+5\leq j \leq n-2$$
        $$(E_{k+2,j},K_{k,j}),\quad 2\leq k\leq n-7,\quad k+5\leq j \leq n-2$$
        $$(F_{k+1,j},L_{k,j}),\quad 2\leq k \leq n-6, \quad k+4\leq j \leq n-2$$
        $$(F_{k+2,k+4},N_{k,k+3}),\quad 2\leq k \leq n-7$$
        $$(G_{k+1,k+5},O_{k,k+3}),\quad 2\leq k \leq n-7$$
        $$(H_{k+2,j},N_{k,j}),\quad 2\leq k \leq n-7,\quad k+4\leq j\leq n-2$$
        $$(J_{k+1,j},O_{k,j}), \quad 2\leq k \leq n-6, k+4\leq j \leq n-2$$ 
        $$(L_{k+1,k+4},P_{k,k+3}),\quad 2\leq k \leq n-6$$
        $$(M_{k+1,j},P_{k,j}),\quad 2\leq k \leq n-6, \quad k+4\leq j \leq n-2$$
    
    After these matchings, the unmatched faces are:
    $$I_{n-6,n-2}=[n] \setminus \{1,n-6,n-4,n-2,n-1\}$$
    $$K_{n-7,n-3}=[n]\setminus \{1,n-7,n-5,n-3,n-1\}$$
    $$K_{n-6,n-3}=[n]\setminus \{1,n-6,n-4,n-3,n-1\}$$
    $$K_{n-6,n-2}=[n]\setminus \{1,n-6,n-4,n-2,n\}$$
    $$K_{n-5,n-2}=[n]\setminus \{1,n-5,n-3,n-2,n\}$$
    $$N_{n-6,n-3}=[n]\setminus \{1,n-6,n-4,n-2,n\}$$
    $$N_{n-6,n-2}=[n]\setminus \{1,n-6,n-4,n-2,n-1,n\}$$
    $$N_{n-5,n-2}=[n] \setminus \{1,n-5,n-3,n-2,n-1,n\}$$
    $$O_{n-6,n-3}=[n] \setminus \{1,n-6,n-5,n-4,n-3,n-1\}$$
    $$O_{n-5,n-2}=[n] \setminus \{1,n-5,n-4,n-3,n-2,n\}$$
    $$P_{n-5,n-2}=[n] \setminus \{1,n-5,n-4,n-3,n-2,n-1,n\}$$
Now we apply matchings $\mathcal{M}_i$ for $i=n-5,...,n-1$, which gives the pairings $(I_{n-6,n-2},N_{n-6,n-3}), \\(K_{n-6,n-3},O_{n-6,n-3}),\quad (K_{n-6,n-2},N_{n-6,n-2}), \quad (K_{n-5,n-2},O_{n-5,n-2}), \quad (N_{n-5,n-2},P_{n-5,n-2})$.

By Theorem \ref{element-acyclic1} , the union of element matchings $\mathcal{M}_1\cup \cdots\cup \mathcal{M}_{n-1}$ is an acyclic matching on the face poset of the complex $\Delta_{3}^{t}(W_n)$. There is one critical cell $K_{n-7,n-3}$ and one $0$-cell(matched to the empty set). Since the set $K_{n-7,n-3}$ has cardinality $n-5$, the complex $\Delta_{3}^{t}(W_n)$ is homotopy equivalent to a sphere of dimension $n-6$.
\end{proof}
\section{The general case ($k\geq 3$)}\label{mainsec2}
In this section we will prove Theorem \ref{zjq-syf-n=3k+1}. To smoothify the dicussion of faces in the matching process, we introduce a unique representation that will be used in the proofs of the next two theorems. For any $\Delta_k^t(W_n)$, we first apply an element matching $\mathcal{M}_1$ using the vertex $1$. After matching $\mathcal{M}_1$, the unmatched faces $\sigma$ are those satisfying the follows.
    $$ 1 \notin \sigma, \sigma \in \Delta_k^t(W_n), \sigma \cup \{1\} \notin \Delta_k^t(W_n)$$
    For any such $\sigma$, we introduce the following representation.
    $$[n]\setminus \sigma = B_0 \sqcup B_1 \sqcup \cdots \sqcup B_{k-1}$$
    where for $0 \leq j \leq k-1, \{i_j\} \subseteq B_j \subseteq \{i_j, i_{j}+1, i_{j}+2\}$ and $i_0 = 1$. $\{1, i_0, \cdots, i_{k-1}\}$ is all the minimum numbers of the face that form a $k$-independent set in $W_n$, which means for any $1 \leq j \leq k-1$, $i_j \geq i_{j-1}+3$ and $i_{k-1} \leq n-2$.
    We denote $B_j$ as $\binom{I_j}{i_j}$, where $ \{0\} \subseteq I_j \subseteq \{0, 1, 2\}$ and $B_j = \{i_j + t | t \in I_j\}$. 
    
    Since $ 1 \notin \sigma, \sigma \in \Delta_k^t(W_n), \sigma \cup \{1\} \notin \Delta_k^t(W_n)$, we can classify $[n] \setminus \sigma$ as follows.

    For $0 \leq j \leq k-3$,
    \begin{itemize}
        \item If $I_j = \{0\}$, then for any $j+1 \leq t \leq k-1$, $i_t \geq i_{t-1} + 3$ and $i_{k-1} \leq n-2$ satisfy the condition.
        \item If $I_j = \{0, 1\}$, when there exists $0 \leq l \leq j-1$ such that $I_l = \{0\}$ or exists $1 \leq l \leq j$ such that $B_{l-1} = {0,2 \choose i_{l-1}}$ and $B_l = {0, 1 \choose i_{l-1}+3}$, then it is the same as the previous case (denoted as ($\star$)). Otherwise, $i_{j+1} = i_j +3$.
        \item If $I_j = \{0, 2\}$, when there exists $0 \leq l \leq j-1$ such that $I_l = \{0\}$ or exists $1 \leq l \leq j-1$ such that $B_{l-1} = {0,2 \choose i_{l-1}}$ and $B_l = {0, 1 \choose i_{l-1}+3}$, then it is the $\star$. Otherwise, $i_{j+1} = i_j +3$ or $i_{j+1} = i_j +4$.
        \item If $I_j = \{0, 1, 2\}$, when there exists $0 \leq l \leq j-1$ such that $I_l = \{0\}$ or exists $1 \leq l \leq j-1$ such that $B_{l-1} = {0,2 \choose i_{l-1}}$ and $B_l = {0, 1 \choose i_{l-1}+3}$, then it is the $\star$; when there exists $1\leq m \leq j-1$, such that $B_m={0,2 \choose i_m}$ and for any $1\leq t \leq j-1$, $i_{t+1}=i_t+3, B_t={0,2\choose i_t}$ or ${0,1,2 \choose i_t}$, then $i_{j+1}=i_{j}+3$ or $i_j+4$. Otherwise, $i_{j+1} = i_j +3$.
    \end{itemize}

    For $j = k-2$,
    \begin{itemize}
        \item If $I_{k-2} = \{0\}$, then it is the $\star$.
        \item If $I_{k-2} = \{0, 1\}$, when there exists $0 \leq l \leq k-3$ such that $I_l = \{0\}$ or exists $1 \leq l \leq j-1$ such that $B_{l-1} = {0,2 \choose i_{l-1}}$ and $B_l = {0, 1 \choose i_{l-1}+3}$, then it is the $\star$. Otherwise, $B_{k-1} = {0 \choose i_{k-2}+3}$.
        \item If $I_{k-2} = \{0, 2\}$, when there exists $0 \leq l \leq k-3$ such that $I_l = \{0\}$ or exists $1 \leq l \leq k-3$ such that $B_{l-1} = {0,2 \choose i_{l-1}}$ and $B_l = {0, 1 \choose i_{l-1}+3}$, then it is the $\star$. Otherwise, when there exists ${0, 1 \choose i_{k-2}-3}$ or ${0, 2 \choose i_{k-2}-4}$ or ${0, 1, 2 \choose i_{k-2}-4}$, then $B_{k-1} = {0 \choose i_{k-2}+3}$ or ${0, 1 \choose i_{k-2}+3}$ or ${0 \choose i_{k-2}+4}$.
        \item If $I_{k-2} = \{0, 1, 2\}$, when there exists $0 \leq l \leq k-3$ such that $I_l = \{0\}$ or exists $1 \leq l \leq k-3$ such that $B_{l-1} = {0,2 \choose i_{l-1}}$ and $B_l = {0, 1 \choose i_{l-1}+3}$, then it is the $\star$; otherwise, when there exists ${0, 1 \choose i_{k-2}-3}$ or ${0, 2 \choose i_{k-2}-4}$ or ${0, 1, 2 \choose i_{k-2}-4}$, then $B_{k-1} = {0 \choose i_{k-2}+3}$.
        
        Condition 1: there doesn't exist $0 \leq l \leq k-3$ such that $I_l = \{0\}$ or exists $1 \leq l \leq k-3$ such that $B_{l-1} = {0,2 \choose i_{l-1}}$ and $B_l = {0, 1 \choose i_{l-1}+3}$. 
        
        Condition 2: there exists $1\leq j \leq k-3$, such that $B_j={0,2\choose i_j}$ and for any $j\leq t \leq k-3$, $i_{t+1}=i_{t}+3$, $B_t={0,2 \choose i_t}$ or ${0,1,2 \choose i_{t}}$ and $B_{t+1}={0,2 \choose i_{t+1}}$ or ${0,1,2 \choose i_{t+1}}$; or else $i_{t+1}=i_t+4$, $B_t={0,2\choose i_t}$ or ${0,1,2 \choose i_t}$ and $B_{t+1}={0,2 \choose i_{t+1}}$.

        \item If $I_{k-2}=\{0,2\}$, when $B_0={0,1 \choose 1}$ and it satisfies the condition 1,2, $n-2=i_{k-2}+3$, then $B_{k-1} = {0, 1 \choose i_{k-2}+3}$ or ${0 \choose i_{k-2}+3} $ or ${0,1,2\choose i_{k-2}+3}$ or ${0,2\choose i_{k-2}+3}$.

         \item If $I_{k-2}=\{0,2\}$, when $B_0={0,1 \choose 1}$ and it satisfies the condition 1,2, $n-2=i_{k-2}+4$, then $B_{k-1} = {0, 1 \choose i_{k-2}+3}$ or ${0 \choose i_{k-2}+3} $ or ${0,2\choose i_{k-2}+4}$ or ${0 \choose i_{k-2}+4}$.
         
        \item If $I_{k-2}=\{0,1,2\}$, when $B_0={0,1\choose 1}$ and it satisfies the condition 1,2, $n-2=i_{k-2}+3$, then $B_{k-1} = {0, 1 \choose i_{k-2}+3}$ or ${0 \choose i_{k-2}+3} $ or ${0,1,2\choose i_{k-2}+3}$ or ${0,2\choose i_{k-2}+3}$.
        \item If $I_{k-2}=\{0,1,2\}$, when $B_0={0,1\choose 1}$ and it satisfies the condition 1,2, $n-2=i_{k-2}+4$, then $B_{k-1} = {0, 1 \choose i_{k-2}+3}$ or ${0 \choose i_{k-2}+4}$ or ${0 \choose i_{k-2}+3} $ or ${0,2\choose i_{k-2}+4}$.
    \end{itemize}
\begin{example}
    When $n=14$ and $k=4$, here is one unmatched face $\sigma$ after $\mathcal{M}_1$.
    $$\sigma=\{2,6,8,10,13,14\}= [n]\setminus \{1,3,4,5,7,9,11,12\}={0,2\choose 1}{0,1\choose 4}{0,2 \choose 7}{0,1\choose 11}$$
\end{example}
\begin{theorem}\label{zjq-n=3k+1}
    For $k\geq 3$, we have:
    \begin{equation}
        \Delta_k^t(W_{3k+1}) \simeq \mathbb{S}^3
    \end{equation}
\end{theorem}
\begin{proof}
    
    When $n=3k+1(k \geq 3)$, for any $1\leq j \leq k-1$, if $i_{j-1}=3(j-1)+2$, then $i_j=3j+2$. Otherwise, $i_j=3j+1$ or $3j+2$.

    We proceed with $\mathcal{M}_{2}, \mathcal{M}_{3}, \cdots, \mathcal{M}_{7}$ in order.

    For $\mathcal{M}_2$, we have the following matchings:
    \begin{itemize}
        \item ${0,1 \choose 1}{0 \choose 4} \cdots \leftrightarrow {0 \choose 1}{0 \choose 4} \cdots (i_2 = 7, 8)$.
        \item ${0,1 \choose 1}{0, 1 \choose 4} \cdots \leftrightarrow {0 \choose 1}{0, 1 \choose 4} \cdots (i_2 = 7)$. However, we have some faces in the latter form still unmatched, denoted by $*_7(1)$.
        \item ${0,1 \choose 1}{0, 2 \choose 4} \cdots \leftrightarrow {0 \choose 1}{0, 2 \choose 4} \cdots (i_2 = 7, 8)$. Similarly, let $*_7(2), *_8(2)$ denote the unmatched faces of the latter form.
        \item ${0,1 \choose 1}{0, 1, 2 \choose 4} \cdots \leftrightarrow {0 \choose 1}{0, 1, 2 \choose 4} \cdots (i_2 = 7)$. Let $*_7(3)$ denote the unmatched faces of the latter form.
        \item ${0,1,2 \choose 1}{0 \choose 4} \cdots \leftrightarrow {0, 2 \choose 1}{0 \choose 4} \cdots (i_2 = 7, 8)$.
        \item ${0,1,2 \choose 1}{0, 1 \choose 4} \cdots \leftrightarrow {0, 2 \choose 1}{0, 1 \choose 4} \cdots (i_2 = 7)$. Let $*_7(4)$ denote the unmatched faces of the latter form.
        \item ${0,1,2 \choose 1}{0, 2 \choose 4} \cdots \leftrightarrow {0, 2 \choose 1}{0, 2 \choose 4} \cdots (i_2 = 7, 8)$.
        \item ${0,1,2 \choose 1}{0, 1, 2 \choose 4} \cdots \leftrightarrow {0, 2 \choose 1}{0, 1, 2 \choose 4} \cdots (i_2 = 7)$. Let $*_7(5)$ denote the unmatched faces of the latter form.
    \end{itemize}

    For $\mathcal{M}_3$, we have the following matchings:
    \begin{itemize}
        \item ${0, 2 \choose 1}{0, 1 \choose 4} \cdots \leftrightarrow {0 \choose 1}{0, 1 \choose 4} \cdots (i_2 = 7,8)$. And $*_7(4)$ is all matched with $*_7(1)$ in the $i_2 = 7$ case.
        \item ${0, 2 \choose 1}{0, 1, 2 \choose 4} \cdots \leftrightarrow {0 \choose 1}{0, 1, 2 \choose 4} \cdots (i_2 = 8)$. Let $*_8(5)$ denote the unmatched faces of the latter form.
        \item ${0, 2 \choose 1}{0 \choose 5} \cdots \leftrightarrow {0 \choose 1}{0 \choose 5} \cdots (i_2 = 8)$.
        \item ${0, 2 \choose 1}{0, 1 \choose 5} \cdots \leftrightarrow {0 \choose 1}{0, 1 \choose 5} \cdots (i_2 = 8)$. Let $*_8(6)$ denote the unmatched faces of the latter form.
        \item ${0, 2 \choose 1}{0, 2 \choose 5} \cdots \leftrightarrow {0 \choose 1}{0, 2 \choose 5} \cdots (i_2 = 8)$. Let $*_8(7)$ denote the unmatched faces of the latter form.
        \item ${0, 2 \choose 1}{0, 1, 2 \choose 5} \cdots \leftrightarrow {0 \choose 1}{0, 1, 2 \choose 5} \cdots (i_2 = 8)$. Let $*_8(8)$ denote the unmatched faces of the latter form.
        \item ${0, 2 \choose 1}{0, 1, 2 \choose 4} \cdots \leftrightarrow {0 \choose 1}{0, 1, 2 \choose 4} \cdots (i_2 = 7)$, i.e. $*_7(5)$ are matched with some faces of $*_7(3)$. Let $*_7(3)'$ denote the unmatched faces of $*_7(3)$.
    \end{itemize}
    In the last matching, notice that the structure of $*_7(5)$ after $i_2$ does not appear in the fourth matching of $\mathcal{M}_2$. Therefore, all faces of $*_7(5)$ can be matched.

    After $\mathcal{M}_3$, $*_7(3)'$ does not have any faces represented by ${0 \choose 1}{0, 1, 2 \choose 4}{0, 1 \choose 7} \cdots$ and ${0 \choose 1}{0, 1, 2 \choose 4}{0 \choose 7}$ $\cdots$. It consists of faces represented by ${0 \choose 1}{0, 1, 2 \choose 4}{0, 1, 2 \choose 7} \cdots$ and ${0 \choose 1}{0, 1, 2 \choose 4}{0, 2 \choose 7} \cdots$.(i.e. If we replace the $B_i$$(i\geq 2)$s of any unmatched face of $*_7(3)'$ to the bottom of ${0,1\choose 1}{0,1,2\choose 4}$ or ${0,2\choose 1}{0,1,2 \choose 4}$, then the faces created are already matched in $\mathcal{M}_1$.) We denote the latter kind as $*_7(9)$.

    For $\mathcal{M}_4$, we have the following matchings:
    \begin{itemize}
        \item ${0 \choose 1}{0, 1, 2 \choose 4} \cdots \leftrightarrow {0 \choose 1}{0, 1 \choose 5} \cdots$, i.e. all faces of $*_8(5)$ are matched with all faces of $*_8(6)$.
        \item ${0 \choose 1}{0, 1, 2 \choose 4}{0, 1, 2 \choose 7}{* \choose 11} \cdots \leftrightarrow {0 \choose 1}{0, 1, 2 \choose 5}{0, 1 \choose 8} {*\choose 11}\cdots$ (for $j \geq 3$, $I_j \neq \{0\}$), i.e. some faces of $*_7(3)'$ are matched with some faces of $*_8(8)$.
        \item ${0 \choose 1}{0, 1, 2 \choose 4}{0, 1, 2 \choose 7}{0, 1, 2 \choose 10}\cdots{0,1,2 \choose 3t+1}{* \choose 3t+5}\cdots \leftrightarrow {0 \choose 1}{0, 1, 2 \choose 5}{0, 1, 2 \choose 8} \cdots {0,1 \choose 3t+2}{* \choose 3t+5}\cdots$ (for $3 \leq t \leq k-2$ and $j \geq t+1$, $I_j \neq \{0\}$).
        \item ${0 \choose 1}{0,1,2 \choose 4}{0,1,2\choose 7}\cdots {0,1,2\choose 3k-2}\leftrightarrow {0 \choose 1}{0,1,2 \choose 5}{0,1,2 \choose 8}\cdots {0,1\choose 3k-1}$
    \end{itemize}
    The unmatched faces of $*_7(3)'$ with a beginning of ${0 \choose 1}{0, 1, 2 \choose 4}{0, 1, 2 \choose 7}$ are 
    \begin{itemize}
        \item ${0 \choose 1}{0, 1, 2 \choose 4}{0, 1, 2 \choose 7}{I_3 \choose 10}\cdots{I_t \choose 3t+1}\cdots{I_{k-1} \choose 3k-2}$, i.e. for $3 \leq t \leq k-1$, $i_t = 3t+1, I_t = \{0, 1, 2\} or \{0, 2\}$ and there must exist some $3 \leq j \leq k-1$, s.t. $I_j=\{0,2\}$. Denote these unmatched faces as $*_7(10)$.
        \item ${0 \choose 1}{0, 1, 2 \choose 4}{0, 1, 2 \choose 7}{I_3 \choose 10}\cdots{I_t \choose 3t+1}{I_{t+1} \choose 3t+5}\cdots{I_{k-1} \choose 3k-1}$, i.e. for $3 \leq j \leq t \leq k-2$, $i_j = 3j+1, I_j = \{0, 1, 2\} \text{ or } \{0, 2\}$ and there must exist some $3 \leq l \leq t$, s.t. $I_l=\{0,2\}$. Denote these unmatched faces as $*_7(11)$.
    \end{itemize}
    The unmatched faces of $*_8(8)$ are ${0 \choose 1}{0, 1, 2 \choose 5}{I_2 \choose 8}\cdots {I_{k-1}\choose 3k-1}$, i.e. for $2\leq j \leq k-1$, $I_j=\{0,1,2\}$ or $\{0,2\}$. Denote these unmatched faces as $*_8(12)$.

    For $\mathcal{M}_5$, we have the following matchings:
    \begin{itemize}
        \item ${0 \choose 1}{0, 1, 2 \choose 4}{0, 2 \choose 7} \cdots \leftrightarrow {0 \choose 1}{0, 2 \choose 4}{0, 2 \choose 7} \cdots$, i.e. all faces of $*_7(9)$ are matched with some faces of $*_7(2)$.
        \item ${0 \choose 1}{0, 1, 2 \choose 4}{0, 1, 2 \choose 7} \cdots \leftrightarrow {0 \choose 1}{0, 2 \choose 4}{0, 1, 2 \choose 7} \cdots$, i.e. all faces of $*_7(10)$ and $*_7(11)$ are matched with some faces of $*_7(2)$.
    \end{itemize}
    The unmatched faces of $*_7(2)$ are ${0 \choose 1}{0,2 \choose 4}{0,1,2 \choose 7} \cdots {0,1,2 \choose 3t+1}{I_{t+1} \choose 3t+5} \cdots {I_{k-1}\choose 3k-1}$, i.e. for $2 \leq t \leq k-2$ and $2\leq j \leq t$, $I_j=\{0,1,2\}$; $j \geq t+1$, $I_j \neq \{0\}$. Denote these unmatched faces as $*_7(13)$.
    
    For $\mathcal{M}_6$, we have the following matchings:
    \begin{itemize}
        \item ${0 \choose 1}{0, 1, 2 \choose 5} \cdots \leftrightarrow {0 \choose 1}{0, 2 \choose 5} \cdots$, i.e. all faces of $*_8(12)$ are matched with all faces of $ *_8(7)$.
    \end{itemize}
    
    For $\mathcal{M}_7$, we have the following matchings:
    \begin{itemize}
        \item ${0 \choose 1}{0,2 \choose 4}{0,1,2 \choose 7}\cdots{0,1,2 \choose 3t+1}{I_{t+1} \choose 3t+5}\cdots$ (for $2\leq t \leq k-2$, $2\leq j \leq t$, $I_j=\{0,1,2\}$; $j\geq t+1$, $I_j \neq \{0\}$)$\leftrightarrow {0 \choose 1}{0,2 \choose 4}\cdots{0,1 \choose 3t+2}{I_{t+1} \choose 3t+5}\cdots$ (If $t \geq 3$, for $2\leq j \leq t-1$, $I_{j}=\{0,1,2\}$. For $t+1\leq j$, $I_{j}\neq \{0\}$) i.e. all faces of $*_7(13)$ are matched with some faces of $*_8(2)$.
    \end{itemize}
    And the only unmatched face is $[n]\setminus {0 \choose 1}{0,2 \choose 4}{0,1,2 \choose 8}{0,1,2 \choose 11}\cdots{0,1,2 \choose 3k-1}=\{2,3,5,7\}$. By discrete Morse theory, we have $\Delta_k^t({W_{3k+1}}) \simeq \mathbb{S}^3$.
\end{proof}
The rest of this section is devoted to solving the case when $n=3k+2$ (Theorem \ref{n=3k+2}). The proof of this theorem may seem quite tedious, so We recommend that readers carefully check each step.
\begin{theorem}\label{n=3k+2}
    For $k\geq 3$, we have:
    \begin{equation}
        \Delta_k^t(W_{3k+2}) \simeq \mathbb{S}^5
    \end{equation}
\end{theorem}
\begin{proof}
When $n=3k+2(k \geq 5)$, for any $1\leq j \leq k-1$, if $i_{j-1}=3(j-1)+1$, then $i_j=3j+1$ or $3j+2$ or $3j+3$; if $i_{j-1}=3(j-1)+2$, then $i_j=3j+2$ or $3j+3$; if $i_{j-1}=3(j-1)+3$, then $i_{j}=3j+3$.

We proceed with $\mathcal{M}_{2}, \mathcal{M}_{3}, \cdots, \mathcal{M}_{11}$ in order.

    For $\mathcal{M}_2$, we have the following matchings:
    \begin{itemize}
        \item ${0,1 \choose 1}{0 \choose 4} \cdots \leftrightarrow {0 \choose 1}{0 \choose 4} \cdots (i_2 = 7,8,9)$.
        \item ${0,1 \choose 1}{0, 1 \choose 4} \cdots \leftrightarrow {0 \choose 1}{0, 1 \choose 4} \cdots (i_2 = 7)$. However, we have some faces in the latter form still unmatched, denoted by $*_7(1)$.
        \item ${0,1 \choose 1}{0, 2 \choose 4} \cdots \leftrightarrow {0 \choose 1}{0, 2 \choose 4} \cdots (i_2 = 7, 8)$. Similarly, let $*_7(2), *_8(2)$ denote the unmatched faces of the latter form.
        \item ${0,1 \choose 1}{0, 1, 2 \choose 4} \cdots \leftrightarrow {0 \choose 1}{0, 1, 2 \choose 4} \cdots (i_2 = 7)$. Let $*_7(3)$ denote the unmatched faces of the latter form.
        \item ${0,1,2 \choose 1}{0 \choose 4} \cdots \leftrightarrow {0, 2 \choose 1}{0 \choose 4} \cdots (i_2 = 7,8,9)$.
        \item ${0,1,2 \choose 1}{0, 1 \choose 4} \cdots \leftrightarrow {0, 2 \choose 1}{0, 1 \choose 4} \cdots (i_2 = 7)$. Let $*_7(4)$ denote the unmatched faces of the latter form.
        \item ${0,1,2 \choose 1}{0, 2 \choose 4} \cdots \leftrightarrow {0, 2 \choose 1}{0, 2 \choose 4} \cdots (i_2 = 7, 8)$.
        \item ${0,1,2 \choose 1}{0, 1, 2 \choose 4} \cdots \leftrightarrow {0, 2 \choose 1}{0, 1, 2 \choose 4} \cdots (i_2 = 7)$. Let $*_7(5)$ denote the unmatched faces of the latter form.
    \end{itemize}
    
    For $\mathcal{M}_3$, we have the following matchings:
    \begin{itemize}
        \item ${0, 2 \choose 1}{0, 1 \choose 4} \cdots \leftrightarrow {0 \choose 1}{0, 1 \choose 4} \cdots (i_2 =7, 8,9)$. And $*_7(4)$ is all matched with $*_7(1)$ in the $i_2 = 7$ case.
        \item ${0, 2 \choose 1}{0, 1, 2 \choose 4} \cdots \leftrightarrow {0 \choose 1}{0, 1, 2 \choose 4} \cdots (i_2 = 8)$. Let $*_8(5)$ denote the unmatched faces of the latter form.
        \item ${0, 2 \choose 1}{0 \choose 5} \cdots \leftrightarrow {0 \choose 1}{0 \choose 5} \cdots (i_2 = 8,9)$.
        \item ${0, 2 \choose 1}{0, 1 \choose 5} \cdots \leftrightarrow {0 \choose 1}{0, 1 \choose 5} \cdots (i_2 = 8)$. Let $*_8(6)$ denote the unmatched faces of the latter form.
        \item ${0, 2 \choose 1}{0, 2 \choose 5} \cdots \leftrightarrow {0 \choose 1}{0, 2 \choose 5} \cdots (i_2 = 8,9)$. Let $*_8(7)$ and $*_9(7)$ denote the unmatched faces of the latter form.
        \item ${0, 2 \choose 1}{0, 1, 2 \choose 5} \cdots \leftrightarrow {0 \choose 1}{0, 1, 2 \choose 5} \cdots (i_2 = 8)$. Let $*_8(8)$ denote the unmatched faces of the latter form.
        \item ${0, 2 \choose 1}{0, 1, 2 \choose 4} \cdots \leftrightarrow {0 \choose 1}{0, 1, 2 \choose 4} \cdots (i_2 = 7)$, i.e. $*_7(5)$ are matched with some faces of $*_7(3)$. Let $*_7(3)'$ denote the unmatched faces of $*_7(3)$.
    \end{itemize}

    For $\mathcal{M}_4$, we have the following matchings:
    \begin{itemize}
        \item ${0 \choose 1}{0,1,2 \choose 4} \cdots \leftrightarrow {0 \choose 1}{0,1 \choose 5} \cdots (i_2 = 8)$, i.e. $*_8(6) \leftrightarrow *_8(7)$
        \item ${0 \choose 1}{0,1,2 \choose 4} \cdots \leftrightarrow {0 \choose 1}{0,1 \choose 5} \cdots (i_2 = 9)$
        \item ${0 \choose 1}{0,2 \choose 4} \cdots \leftrightarrow {0 \choose 1}{0 \choose 6} \cdots (i_2 = 9)$
    \end{itemize}
    Consider $*_7(2)$ in $\mathcal{M}_4$, which is related to the following matchings:
    \begin{itemize}
        \item ${0 \choose 1}{0,2 \choose 4}{I_2 \choose 7}{I_3 \choose 10}\cdots {I_{k-1} \choose 3k-2} \leftrightarrow {0 \choose 1}{I_1 \choose 6}{I_2 \choose 9}{I_3 \choose 12}\cdots {0 \choose 3k}$. \\For the left side, when $j \geq 2$, $I_j = \{0,2\}$ or $\{0,1,2\}$. 
        \\For the right side, when $1 \leq j \leq k-2$, $I_j = \{0,1\}$ or $\{0,1,2\}$.
        \item ${0 \choose 1}{0,2 \choose 4}{I_2 \choose 7}\cdots {I_{t-1} \choose 3t-2}{I_t \choose 3t+2}\cdots{I_{k-1} \choose 3k-1} \leftrightarrow {0 \choose 1}{I_1 \choose 6}{I_2 \choose 9}\cdots {0,2 \choose 3t}{I_t \choose 3t+3} \cdots {I_{k-1} \choose 3k}$. 
        \\Here $3 \leq t \leq k-1$. 
        \\For the left side, $i_t=3t+2$, $i_{t-1}=3t-2$, when $2 \leq j \leq t-1$, $I_j = \{0,2\}$ or $\{0,1,2\}$ and when $t \leq j \leq k-1$, $I_j = \{0,1\}$ or $\{0,1,2\}$. \\For the right side, when $1 \leq j \leq t-2$, $I_j = \{0,1\}$ or $\{0,1,2\}$ and when $t \leq j \leq k-2$, $I_j =\{0,2\}$ or $\{0,1,2\}$ and $I_{k-1} = \{0\}$ or $\{0,1\}$ 
        \item ${0 \choose 1}{0,2 \choose 4}{I_2 \choose 7} \cdots {I_{t-1} \choose 3t-2}{I_t \choose 3t+2} \cdots {I_{m-1} \choose 3m-1}{I_m \choose 3m+3} \cdots {I_{k-1} \choose 3k} $\\$\leftrightarrow {0 \choose 1}{0 \choose 6}{I_2 \choose 9} \cdots {0,2 \choose 3t}{I_t \choose 3t+3} \cdots {I_{m-1} \choose 3m}{I_m \choose 3m+3} \cdots {I_{k-1} \choose 3k}$. \\Here $3\leq t < m\leq k-1$.\\ For the left side, $i_t=3t+2$, $i_{t-1}=3t-2$, $i_{m-1}=3m-1$, $i_m=3m+3$, when $2 \leq j \leq t-1$, $I_j = \{0,2\}$ or $\{0,1,2\}$; when $t \leq j \leq m-1$, $I_j = \{0,1\}$ or $\{0, 1, 2\}$. 
        \\For the right side, when $1 \leq j \leq t-2$, $I_j = \{0,1\}$ or $\{0,1,2\}$; when $t \leq j \leq m-2$(we don't consider this situation if $m=t+1$), $I_j = \{0,2\}$ or $\{0,1,2\}$; $I_{m-1} = \{0\}$ or $\{0,1\}$; when $m \leq j \leq k-1$, $I_j$ is the same as the left side.
        \item ${0 \choose 1}{0,2 \choose 4}{I_2 \choose 7} \cdots {I_{t-1} \choose 3t-2} {I_t \choose 3t+3} \cdots {I_{k-1} \choose 3k} \leftrightarrow {0 \choose 1}{I_1 \choose 6}{I_2 \choose 9} \cdots {0 \choose 3t} {I_t \choose 3t+3} \cdots {I_{k-1} \choose 3k}$. \\Here $3\leq t \leq k-1$. 
        \\For the left side, $i_t=3t+3$, $i_{t-1}=3t-2$, when $2\leq j \leq t-1$, $I_j=\{0,2\}$ or $\{0,1,2\}$; when $k-1\geq j \geq t$, $I_j$ can be any subset of $\{0,1,2\}$. 
        \\For the right side, when $1\leq j \leq t-2$, $I_j=\{0,1\}$ or $\{0,1,2\}$; $I_{t-1}=\{0\}$; when $k-1\geq j \geq t$, $I_{j}$ is the same as the left side.
    \end{itemize}
    Consider $*_8(2)$ in $\mathcal{M}_4$, which is related to the following matchings:
    \begin{itemize}
        \item ${0 \choose 1}{0,2 \choose 4}{I_2 \choose 8}{I_3 \choose 11} \cdots {I_{k-1} \choose 3k-1} \leftrightarrow {0 \choose 1} {0, 2 \choose 6}{I_2 \choose 9}\cdots {I_{k-1} \choose 3k}$. 
        \\For the left side, when $2 \leq j \leq k-1$, $I_j = \{0, 1\}$ or $\{0, 1, 2\}$. 
        \\For the right side, when $2 \leq  j \leq k-2$, $I_j = \{0,2\}$ or $\{0,1,2\}$; $I_{k-1}= \{0\}$ or $\{0,1\}$.
        \item ${0 \choose 1}{0,2 \choose 4}{I_2 \choose 8}\cdots {I_{t-1} \choose 3t-1}{I_t \choose 3t+3}\cdots{I_{k-1} \choose 3k} \leftrightarrow {0 \choose 1}{0,2 \choose 6}{I_2 \choose 9}\cdots {I_{t-1} \choose 3t}{I_t \choose 3t+3} \cdots {I_{k-1} \choose 3k}$. \\Here $3\leq t \leq k-1$. 
        \\For the left side, $i_t=3t+3$, $i_{t-1}=3t-1$, when $2 \leq  j \leq t-1$, $I_j = \{0,1\}$ or $\{0,1,2\}$. 
        \\For the right side, when $2\leq j \leq t-2$, $I_j=\{0,2\}$ or $\{0,1,2\}$; $I_t = \{0\}$ or $\{0,1\}$; when $t \leq j \leq k-1$, $I_j$ is the same as the left side.
    \end{itemize}
    Consider $*_7(3)'$ in $\mathcal{M}_4$, which is related to the following matchings:
    \begin{itemize}
        \item ${0 \choose 1}{0,1,2 \choose 4}{0,1,2 \choose 7} \cdots {0,1,2 \choose 3k-2} \leftrightarrow {0 \choose 1}{0,1,2 \choose 5}{0,1,2 \choose 8} \cdots {0,1 \choose 3k-1}$.
        \item ${0 \choose 1}{0,1,2 \choose 4}{I_2 \choose 7}\cdots{0,2 \choose 3t+1}\cdots{I_{k-1} \choose 3k-2} \leftrightarrow {0 \choose 1}{0,1,2 \choose 5}\cdots {I_t \choose 3t+3}\cdots{0 \choose 3k}$. 
        \\Here $2\leq t \leq k-1$. 
        \\For the left side, when $2\leq j \leq t-1$, $I_j = \{0,1,2\}$(we don't consider this situation if $t=2$); $I_t=\{0,2\}$; when $t+1 \leq j \leq k-1$, $I_j = \{0,2\}$ or $\{0,1,2\}$(we don't consider this situation if $t=k-1$). 
        \\For the right side, when $2\leq j \leq t-1$, $I_j = \{0,1,2\}$(we don't consider this situation if $t=2$); when $t \leq j \leq k-2$, $I_j = \{0,1\}$ or $\{0,1,2\}$(we don't consider this situation if $t=k-1$); $I_{k-1}=\{0\}$.
        \item ${0 \choose 1}{0,1,2 \choose 4}{0,1,2 \choose 7}\cdots {0,1,2 \choose 3t-2}{I_t \choose 3t+2}\cdots {I_{k-1} \choose 3k}\leftrightarrow {0 \choose 1}{0,1,2 \choose 5}{I_2 \choose 8}\cdots{0,1 \choose 3t-1}{I_t \choose 3t+2}\cdots{I_{k-1} \choose 3k-1}$. 
        \\Here $3\leq t \leq k-1$. 
        \\For the left side, $i_t=3t+2$, $i_{t-1}=3t-2$, when $2\leq j \leq t-1$, $I_j=\{0,1,2\}$; when $j\geq t$, $I_j$ can be any subset of $\{0,1,2\}$ that satisfies the condition of $*_7(3)'$. 
        \\For the right side, when $2\leq j \leq t-2$, $I_j=\{0,1,2\}$(we don't consider this situation if $t=3$); $I_{t-1}=\{0,1\}$; when $j\geq t$, $I_j$ is the same as the left side.
        \item ${0 \choose 1}{0,1,2 \choose 4}\cdots{0,2 \choose 3m+1} \cdots {I_{t-1} \choose 3t-2}{I_t \choose 3t+2}\cdots{I_{k-1} \choose 3k-1} \\\leftrightarrow {0 \choose 1}{0,1,2 \choose 5}\cdots{I_m \choose 3m+3}\cdots{0,2 \choose 3t}{I_t\choose 3t+3}\cdots {I_{k-1} \choose 3k}$. 
        \\Here $2\leq m <t \leq k-1$. 
        \\For the left side, $i_t=3t+2$, $i_{t-1}=3t-2$, when $2\leq j \leq m-1$, $I_j=\{0,1,2\}$(we don't consider this situation if $m=2$); when $m+1 \leq j \leq t-1$, $I_j = \{0,2\}$ or $\{0,1,2\}$(we don't consider this situation if $m=t-1$); when $t \leq j \leq k-1$, $I_j = \{0,1\}$ or $\{0,1,2\}$. 
        \\For the right side, when $2\leq j \leq m-1$, $I_j=\{0,1,2\}$(we don't consider this situation if $m=2$); when $m\leq j \leq t-2$, $I_j = \{0,1\}$ or $\{0,1,2\}$(we don't consider this situation if $m=t-1$); $I_{t-1}=\{0,2\}$; when $t \leq j \leq k-2$, $I_j = \{0,1,2\}$ or $\{0,2\}$(we don't consider this situation if $t=k-1$); $I_{k-1} = \{0\}$ or $\{0,1\}$.
        \item ${0 \choose 1}{0,1,2 \choose 4}{0,1,2 \choose 7}\cdots{0,1,2\choose 3t-2}{I_t \choose 3t+2}\cdots {I_{m-1}\choose 3m-1}{I_m \choose 3m+3}\cdots {I_{k-1}\choose 3k} 
        \\\leftrightarrow {0\choose 1}{0,1,2\choose 5}\cdots {0,1\choose 3t-1}{I_t \choose 3t+2}\cdots {I_{m-1}\choose 3m-1}{I_m\choose 3m+3}\cdots {I_{k-1}\choose 3k}$. 
        \\Here $3\leq t <m \leq k-1$. 
        \\For the left side, $i_t=3t+2$, $i_{t-1}=3t-2$, $i_m=3m+3$, $i_{m-1}=3m-1$, when $2\leq j \leq t-1$, $I_j=\{0,1,2\}$; when $t\leq j \leq m-1$, $I_j=\{0,1\}$ or $\{0,1,2\}$; when $m\leq j \leq k-1$, $I_j$ can be any subset of $\{0,1,2\}$. 
        \\For the right side, when $2\leq j \leq t-2$, $I_j=\{0,1,2\}$(we don't consider this situation if $t=3$); $I_{t-1}=\{0,1\}$; when $j\geq t$, $I_j$ is the same as the left side.
        \item ${0 \choose 1}{0,1,2\choose 4}\cdots {0,2 \choose 3l+1}\cdots {I_{t-1}\choose 3t-2}{I_t \choose 3t+2}\cdots{I_{m-1}\choose 3m-1}{I_m \choose 3m+3}\cdots {I_{k-1}\choose 3k} \\\leftrightarrow {0\choose 1}{0,1,2\choose 5}\cdots {0,1,2 \choose 3l-1}{I_l \choose 3l+3}\cdots {0,2 \choose 3t}{I_t \choose 3t+3}\cdots{I_{m-1}\choose 3m}{I_m \choose 3m+3}\cdots {I_{k-1}\choose 3k}$. 
        \\Here $2 \leq l <t<m \leq k-1$. \\For the left side, $i_t=3t+2$, $i_{t-1}=3t-2$, $i_{m-1}=3m-1$, $i_m=3m+3$, when $2\leq j\leq l-1$, $I_j=\{0,1,2\}$(we don't consider this situation if $l=2$); $I_l=\{0,2\}$; when $l+1\leq j \leq t-1$, $I_j=\{0,1,2\}$ or $\{0,2\}$(we don't consider this situation if $t=l+1$); when $t\leq j \leq m-1$, $I_j=\{0,1\}$ or $\{0,1,2\}$; when $j\geq m$, $I_j$ can be any subset of $\{0,1,2\}$. 
        \\For the right side, when $1\leq j \leq l-1$, $I_j=\{0,1,2\}$; when $l\leq j \leq t-2$, $I_j=\{0,1\}$ or $\{0,1,2\}$(we don't consider this situation if $t=l+1$); $I_{t-1}=\{0,2\}$; when $t\leq j \leq m-2$, $I_j=\{0,1,2\}$ or $\{0,2\}$; $I_{m-1}=\{0\}$ or $\{0,1\}$; when $m\leq j \leq k-1$, $I_j$ is the same as the left side.
        \item ${0 \choose 1}{0,1,2 \choose 4}\cdots {0,1,2 \choose 3t-2}{I_t \choose 3t+3}\cdots {I_{k-1}\choose 3k}\leftrightarrow {0 \choose 1}{0,1,2 \choose 5}\cdots {0,1 \choose 3t-1}{I_t \choose 3t+3}\cdots {I_{k-1}\choose 3k}$. \\Here $3\leq t\leq k-1$. 
        \\For the left side, $i_t=3t+3$, $i_{t-1}=3t-2$, when $2\leq j \leq t-1$, $I_j=\{0,1,2\}$; when $t \leq j \leq k-1$, $I_j$ can be any subset of $\{0,1,2\}$. 
        \\For the right side, when $1\leq j \leq t-2$, $I_j=\{0,1,2\}$; $I_{t-1}=\{0,1\}$; when $t\leq j \leq k-1$, $I_j$ is the same as the left side.
        \item ${0 \choose 1}{0,1,2 \choose 4}\cdots{0,2 \choose 3m+1}\cdots{I_{t-1}\choose 3t-2}{I_t \choose 3t+3}\cdots {I_{k-
        1}\choose 3k} \\\leftrightarrow {0\choose 1}{0,1,2\choose 5}\cdots {0,1,2 \choose 3m-1}{I_m \choose 3m+3}\cdots {0 \choose 3t}{I_t\choose 3t+3} \cdots {I_{k-1}\choose 3k}$. 
        \\Here $2\leq m <t \leq k-1$. 
        \\For the left side, $i_t=3t+3$, $i_{t-1}=3t-2$, when $2\leq j \leq m-1$, $I_j=\{0,1,2\}$(we don't consider this situation if $m=2$); $I_m=\{0,2\}$; when $m+1\leq j \leq t-1$, $I_j=\{0,2\}$ or $\{0,1,2\}$; if $t\leq j \leq k-1$, $I_j$ can be any subset of $\{0,1,2\}$. 
        \\For the right side, when $1\leq j \leq m-1$, $I_j=\{0,1,2\}$; when $m\leq j \leq t-2$, $I_j=\{0,1\}$ or $\{0,1,2\}$; $I_{t-1}=\{0\}$; when $t \leq j \leq k-1$, $I_j$ is the same as the left side.
        
    \end{itemize}
    After $\mathcal{M}_4$, the unmatched faces of $*_7(2)$ are :
    \begin{itemize}
        \item ${0\choose1}{0,2\choose 4}{I_2\choose 7}\cdots {I_{t-1}\choose 3t-2}{I_t\choose 3t+2}\cdots {I_{k-1}\choose 3k-1}$, \\i.e. $3\leq t \leq k-1$, $i_t=3t+2$, $i_{t-1}=3t-2$;\\ when $2 \leq j \leq t-1$, $I_j=\{0,2\}$ or $\{0,1,2\}$;\\ there exists $m$, such that $t\leq m \leq k-1$, $I_m=\{0,2\}$;\\when $t\leq j \leq k-1$, $I_j$ can be any subset of $\{0,1,2\}$ that satisfies the condition of $*_7(2)$.
        \item ${0\choose1}{0,2\choose 4}{I_2\choose 7}\cdots {I_{t-1}\choose 3t-2}{I_t\choose 3t+2}\cdots {I_{m-1}\choose 3m-1}{I_m \choose 3m+3}\cdots {I_{k-1}\choose 3k}$,
        \\i.e. $3\leq t<m \leq k-1$, $i_t=3t+2$, $i_{t-1}=3t-2$, $i_{m-1}=3m-1$, $i_m=3m+3$;
        when $2\leq j \leq t-1$, $I_j=\{0,2\}$ or $\{0,1,2\}$;
        \\there exists $l$, such that $t\leq l \leq m-1$, $I_l=\{0,2\}$;
        \\when $t\leq j \leq k-1$, $I_j$ can be any subset of $\{0,1,2\}$ that satisfies the condition of $*_7(2)$.
    \end{itemize}
    The unmatched faces with a beginning of ${0\choose 1}{0,1\choose 6}{I_2\choose 9}$, denoted as $*_{6,9}(9)$, actually are:
    \begin{itemize}
        \item ${0\choose 1}{0,1\choose 6}{I_2\choose 9}\cdots {I_{k-1}\choose 3k}$,
        \\i.e. when $1\leq j \leq k-2$, $I_j=\{0,1\}$ or $\{0,1,2\}$; $I_{k-1}=\{0,1\}$ or $\{0,1,2\}$ or $\{0,2\}$.
        \item ${0\choose 1}{0,1\choose 6}{I_2\choose 9}\cdots {0,2\choose 3t+3}\cdots {I_{k-1}\choose 3k}$,
        \\i.e. $2\leq t \leq k-2$;
        \\when $2\leq j \leq t-1$, $I_j=\{0,1,2\}$ or $\{0,1\}$; 
        \\when $t+1\leq j \leq k-1$, $I_j=\{0,1,2\}$ or $\{0,2\}$.
    \end{itemize}
    The unmatched faces with a beginning of ${0\choose 1}{0,1,2\choose 6}{I_2\choose 9}$, denoted as $*_{6,9}(10)$, actually are:
    \begin{itemize}
        \item ${0\choose 1}{0,1,2\choose 6}{I_2\choose 9}\cdots {I_{k-1}\choose 3k}$,
        \\i.e. when $1\leq j \leq k-2$, $I_j=\{0,1\}$ or $\{0,1,2\}$; $I_{k-1}=\{0,1\}$ or $\{0,1,2\}$ or $\{0,2\}$.
        \item ${0\choose 1}{0,1,2\choose 6}{I_2\choose 9}\cdots {0,2\choose 3t+3}\cdots {I_{k-1}\choose 3k}$,
        \\i.e. $2\leq t \leq k-2$;
        \\when $2\leq j \leq t-1$, $I_j=\{0,1,2\}$ or $\{0,1\}$; 
        \\when $t+1\leq j \leq k-1$, $I_j=\{0,1,2\}$ or $\{0,2\}$.
    \end{itemize}
    The unmatched faces of $*_8(2)$ are:
    \begin{itemize}
        \item ${0\choose 1}{0,2\choose 4}{I_2\choose 8}\cdots {0,2\choose 3m+2}\cdots {I_{k-1}\choose 3k-1}$,
        \\i.e. $2\leq m \leq k-1$;
        \\when $2\leq j \leq m-1$, $I_j=\{0,1\}$ or $\{0,1,2\}$(we don't consider this situation if m=2);
        \\when $m+1\leq j \leq k-1$, $I_j=\{0,1,2\}$ or $\{0,2\}$(we don't consider this situation $m=k-1$).
        \item ${0\choose 1}{0,2\choose 4}{I_2\choose 8}\cdots{0,2\choose 3m+2}\cdots {I_{t-1}\choose 3t-1}{I_t\choose 3t+3}\cdots {I_{k-1}\choose 3k}$,
        \\i.e. when $2\leq j \leq m-1$, $I_j=\{0,1\}$ or $\{0,1,2\}$(we don't consider this situation if m=2);$I_m=\{0,2\}$;
        \\when $m+1\leq j \leq t-1$, $I_j=\{0,1,2\}$ or $\{0,2\}$(we don't consider this situation $m=t-1$);
        \\when $t\leq j \leq k-1$, $I_j$ can be any subset of $\{0,1,2\}$ that satisfies the condition of $*_8(2)$.
    \end{itemize}
    The unmatched faces with a beginning of ${0\choose 1}{0,2\choose 6}{I_2\choose 9}$, denoted as $*_{6,9}(11)$, actually are:
    \begin{itemize}
        \item ${0\choose 1}{0,2\choose 6}{I_2\choose 9}\cdots {I_{k-1}\choose3k}$, i.e. when $2\leq j\leq k-1$, $I_j=\{0,2\}$ or $\{0,1,2\}$.
    \end{itemize}
    The unmatched faces of $*_7(3)'$ are:
    \begin{itemize}
        \item ${0\choose 1}{0,1,2\choose 4}\cdots {0,2\choose3m+1}\cdots {I_{t-1}\choose 3t-2}{I_t\choose 3t+2}\cdots {I_{k-1}\choose 3k-1}$,
        \\i.e. $2\leq m <t \leq k-1$;
        \\when $2\leq j \leq m-1$, $I_j=\{0,1,2\}$(we don't consider this situation if $m=2$); $I_m=\{0,2\}$;
        \\when $m+1\leq j \leq t-1$, $I_j=\{0,2\}$ or $\{0,1,2\}$(we don't consider this situation if $m=t-1$);
        \\there exists $l$, such that $t\leq l \leq k-1$, $I_l=\{0,2\}$;
        \\when $t\leq j \leq k-1$, $I_j$ can be any subset of $\{0,1,2\}$ that satisfies the condition of $*_7(3)'$.
        \item ${0\choose 1}{0,1,2\choose 4}\cdots {I_{t-1}\choose 3t-2}{I_t\choose 3t+2}\cdots{I_{m-1}\choose3m-1}{I_m\choose 3m+3}\cdots  {I_{k-1}\choose 3k}$,
        \\i.e. $3\leq t <m\leq k-1$, $i_t=3t+2$, $i_{t-1}=3t-2$, $i_{m-1}=3m-1$, $i_m=3m+3$;
        \\there exists $n$, such that $2\leq n \leq t-1$, $I_n=\{0,2\}$;
        \\when $2\leq j \leq t-1$, $I_j=\{0,2\}$ or $\{0,1,2\}$;
        \\there exists $l$, such that $t\leq l \leq m-1$, $I_l=\{0,2\}$;
        \\when $t\leq j \leq k-1$, $I_j$ can be any subset of $\{0,1,2\}$ that satisfies the condition of $*_7(3)'$.
    \end{itemize}
    The unmatched faces of $*_8(8)$ and unmatched faces with a beginning of ${0\choose 1}{0,1,2\choose 5}{I_2\choose 9}$, together denoted as $*_5(12)$, actually are:
    \begin{itemize}
        \item ${0\choose 1}{0,1,2\choose 5}{I_2\choose 8}\cdots {I_{k-1}\choose 3k-1}$,
        \\i.e. when $2\leq j \leq k-1$, $I_j=\{0,1,2\}$ or $\{0,2\}$.
        \item ${0\choose 1}{0,1,2\choose 5}\cdots{0,1,2\choose 3t-1}{I_t\choose 3t+3}\cdots{0,2\choose 3m+3}\cdots {I_{k-1}\choose 3k}$,
        \\i.e. $2\leq t\leq m \leq k-1$;
        \\when $1\leq j \leq t-1$, $I_j=\{0,1,2\}$;
        \\when $t\leq j \leq m-1$, $I_j=\{0,1\}$ or $\{0,1,2\}$; $I_m=\{0,2\}$(we don't consider this situation if $m=t$);
        \\when $m+1\leq j \leq k-1$, $I_j=\{0,2\}$ or $\{0,1,2\}$(we don't consider this situation if $m=k-1$).
        \item ${0\choose 1}{0,1,2\choose 5}\cdots{0,1,2\choose 3t-1}{I_t\choose 3t+3}\cdots {I_{k-1}\choose 3k}$,
        \\i.e. $2\leq t \leq k-1$;
        \\when $1\leq j \leq t-1$, $I_j=\{0,1,2\}$;
        \\when $t\leq j \leq k-1$, $I_j=\{0,1\}$ or $\{0,1,2\}$.
        \item ${0\choose 1}{0,1,2\choose 5}\cdots{I_{m-1}\choose 3m-1}{0,2\choose 3m+2}\cdots{I_{t-1}\choose 3t-1}{I_t\choose 3t+3}\cdots {I_{k-1}\choose 3k}$,
        \\i.e. $2\leq m<t \leq k-1$;
        \\when $1\leq j \leq m-1$, $I_j=\{0,1,2\}$ or $\{0,1\}$, $I_m=\{0,2\}$;
        \\when $m+1\leq j \leq t-1$, $I_j=\{0,2\}$ or $\{0,1,2\}$(we don't consider this situation if $m=t-1$);
        \\when $t\leq j \leq k-1$, $I_j$ can be any subset of $\{0,1,2\}$ that satisfies the condition of $*_8(8)$.
    \end{itemize}
    For $\mathcal{M}_5$, we have the following matchings:
    \begin{itemize}
        \item ${0\choose 1}{0,1,2\choose 4}{I_2\choose 7} \cdots \leftrightarrow {0\choose 1}{0,2\choose 4}{I_2\choose 7}\cdots$, i.e. all remaining faces of $*_7(3)'$ are matched with some faces of $*_7(2)$.  
        \item ${0 \choose 1}{0,1,2 \choose 5}{I_2 \choose 9}\cdots \leftrightarrow {0 \choose 1}{0,1\choose 6}{I_2\choose 9}\cdots$, i.e. some faces of $*_5(12)$ are matched with all remaining faces of $*_{6,9}(9)$.
        \item ${0\choose 1}{0,1,2\choose 5}{0,1,2 \choose 8}\cdots {0,1,2 \choose 3k-1}\leftrightarrow {0 \choose 1}{0,1,2\choose 6}{0,1,2 \choose 9}\cdots {0,1,2\choose 3k-3}{0,1 \choose 3k}$.
        \item ${0 \choose 1 }{0,1,2 \choose 5}\cdots {0,1,2 \choose 3t-1}{I_t \choose 3t+3}\cdots {0,2 \choose 3m+3}\cdots {I_{k-1}\choose 3k}\\\leftrightarrow {0 \choose 1}{0,1,2\choose 6}\cdots {0,1\choose 3t}{I_t\choose 3t+3}\cdots {0,3 \choose 3m+3}\cdots {I_{k-1}\choose 3k}$, i.e. some faces of $*_5(12)$ are matched with some faces of $*_{6,9}(10)$.
        \\Here $3\leq t \leq m\leq k-1$. \\For the left side, $i_{t-1}=3t-1$, $i_t=3t+3$, when $1\leq j \leq t-1$, $I_j=\{0,1,2\}$; when $t\leq j \leq m-1 $, $I_j=\{0,1\}$ or $\{0,1,2\}$(we don't consider this situation if $m=t+1$); when $m+1 \leq j \leq k-1$, $I_j=\{0,2\}$ or $\{0,1,2\}$(we don't consider this situation if $m=k-1$). 
        \\For the right side, when $1\leq j \leq t-2$, $I_j=\{0,1,2\}$; $I_{t-1}=\{0,1\}$; when $t\leq j \leq m-1 $, $I_j=\{0,1\}$ or $\{0,1,2\}$(we don't consider this situation if $m=t+1$); when $m+1 \leq j \leq k-1$, $I_j=\{0,2\}$ or $\{0,1,2\}$(we don't consider this situation if $m=k-1$).
        \item ${0 \choose 1}{0,1,2 \choose 5}\cdots {0,1,2 \choose 3t-1}{I_t \choose 3t+3}\cdots {I_{k-1}\choose 3k}\leftrightarrow {0\choose 1}{0,1,2\choose 6}\cdots {0,1\choose 3t}{I_t\choose 3t+3}\cdots {I_{k-1}\choose 3k}$, some faces of $*_5(12)$ are matched with some faces of $*_{6,9}(10)$.
        \\Here $3\leq t\leq k-1 $. 
        \\For the left side, $i_{t-1}=3t-1$, $i_t=3t+3$, when $1\leq j \leq t-1$, $I_j=\{0,1,2\}$; when $t\leq j \leq k-1$, $I_j=\{0,1,2\}$ or $\{0,1\}$. 
        \\For the right side, when $1\leq j \leq t-2$, $I_j=\{0,1,2\}$; $I_t=\{0,1\}$; when $t\leq j \leq k-1$, $I_j=\{0,1,2\}$ or $\{0,1\}$.
    \end{itemize}
    After $\mathcal{M}_5$, the unmatched faces of $*_7(2)$ are:
    \begin{itemize}
        \item ${0\choose1}{0,2\choose 4}{0,1,2\choose 7}\cdots {0,1,2\choose 3t-2}{I_t\choose 3t+2}\cdots {I_{k-1}\choose 3k-1}$, \\i.e. $3\leq t \leq k-1$, $i_t=3t+2$, $i_{t-1}=3t-2$;
        \\ when $2 \leq j \leq t-1$, $I_j=\{0,1,2\}$;
        \\ there exists $m$, such that $t\leq m \leq k-1$, $I_m=\{0,2\}$;
        \\when $t\leq j \leq k-1$, $I_j$ can be any subset of $\{0,1,2\}$ that satisfies the condition of $*_7(2)$.
        \item ${0\choose1}{0,2\choose 4}{0,1,2\choose 7}\cdots {0,1,2\choose 3t-2}{I_t\choose 3t+2}\cdots {I_{m-1}\choose 3m-1}{I_m \choose 3m+3}\cdots {I_{k-1}\choose 3k}$,
        \\i.e. $3\leq t<m \leq k-1$, $i_t=3t+2$, $i_{t-1}=3t-2$, $i_{m-1}=3m-1$, $i_m=3m+3$;
        when $2\leq j \leq t-1$, $I_j=\{0,1,2\}$;
        \\there exists $l$, such that $t\leq l \leq m-1$, $I_l=\{0,2\}$;
        \\when $t\leq j \leq k-1$, $I_j$ can be any subset of $\{0,1,2\}$ that satisfies the condition of $*_7(2)$.
    \end{itemize}
    The unmatched faces of $*_5(12)$ are:
    \begin{itemize}
        \item ${0\choose 1}{0,1,2\choose 5}\cdots{I_{m-1}\choose 3m-1}{0,2\choose 3m+2}\cdots{I_{t-1}\choose 3t-1}{I_t\choose 3t+3}\cdots {I_{k-1}\choose 3k}$,
        \\i.e. $2\leq m<t \leq k-1$;
        \\when $1\leq j \leq m-1$, $I_j=\{0,1,2\}$ or $\{0,1\}$, $I_m=\{0,2\}$;
        \\when $m+1\leq j \leq t-1$, $I_j=\{0,2\}$ or $\{0,1,2\}$(we don't consider this situation if $m=t-1$);
        \\when $t\leq j \leq k-1$, $I_j$ can be any subset of $\{0,1,2\}$ that satisfies the condition of $*_8(8)$.
        \item ${0\choose 1}{0,1,2\choose 5}{I_2\choose 8}\cdots{0,2\choose 3m+2}\cdots {I_{k-1}\choose 3k-1}$,
        \\i.e. $2\leq m \leq k-1$;
        \\when $2\leq j \leq m-1$, $I_j=\{0,1,2\}$; $I_m=\{0,2\}$(we don't consider this situation if $m=2$);
        \\when $m+1\leq j \leq k-1$, $I_j=\{0,1,2\}$ or $\{0,2\}$(we don't consider this situation if $m=k-1$).
    \end{itemize}
    The unmatched faces of $*_{6,9}(10)$ are:
    \begin{itemize}
        \item ${0\choose 1}{0,1,2\choose 6}{I_2\choose 9}\cdots {I_{k-1}\choose3k}$, i.e. when $2\leq j\leq k-1$, $I_j=\{0,2\}$ or $\{0,1,2\}$.
    \end{itemize}
    For $\mathcal{M}_6$, we have the following matchings:
    \begin{itemize}
        \item ${0\choose 1}{0,1,2\choose 5}\cdots \leftrightarrow {0\choose 1}{0,2\choose 5}\cdots$, i.e. all remaining faces of $*_5(12)$ are matched with some faces of $*_8(7)$.
    \end{itemize}
    Afeter $\mathcal{M}_6$, the unmatched faces of $*_8(7)$ are :
    \begin{itemize}
        \item ${0 \choose 1}{0,2\choose 5}{0,1,2\choose 8}\cdots{0,1,2\choose 3k-1}$.
        \item ${0\choose 1}{0,2\choose 5}\cdots {0,1,2\choose 3t-1}{I_{t-1}\choose 3t+3}\cdots{0,2\choose 3m+3}\cdots {I_{k-1}\choose 3k}$,
        \\i.e. $3\leq t\leq m \leq k-1$, $i_t=3t+3$, $i_{t-1}=3t-1$;
        \\when $2\leq j \leq t-1$, $I_j=\{0,1,2\}$;
        \\when $t\leq j \leq m-1$, $I_j=\{0,1\}$ or $\{0,1,2\}$(we don't consider this situation if $m=t$);$I_m=\{0,2\}$
        \\when $m+1\leq j \leq k-1$, $I_{j}=\{0,2\}$ or $\{0,1,2\}$(we don't consider this situation if $m=k-1$).
        \item ${0\choose 1}{0,2\choose 5}\cdots {0,1,2\choose 3t-1}{I_t \choose 3t+3}\cdots {I_{k-1}\choose 3k}$.
        \\i.e. when $1\leq j \leq t-1$, $I_j=\{0,1,2\}$;
        \\when $t\leq j \leq k-1$, $I_j=\{0,1,2\}$ or $\{0,1\}$
    \end{itemize}
    For $\mathcal{M}_7$, we have the following matchings:
    \begin{itemize}
        \item ${0 \choose 1}{0,1,2 \choose 6}{I_2 \choose 9}\cdots{I_{k-1} \choose 3k}\leftrightarrow {0 \choose 1}{0,2 \choose 6}{I_2 \choose 9} \cdots {I_{k-1} \choose 3k}$, i.e. all remaining faces of $*_{6,9}(10)$ are matched with all remaining faces of $*_{6,9}(11)$.
        \item ${0 \choose 1}{0,2 \choose 4}{0,1,2 \choose 7}\cdots{0,1,2 \choose 3t-2}{I_t \choose 3t+2}\cdots{I_{k-1} \choose 3k-1}\leftrightarrow{0 \choose 1}{0,2 \choose 4}{I_2 \choose 8}\cdots{0,1 \choose 3t-1}{I_t \choose 3t+2}\cdots {I_{k-1} \choose 3k-1}$, i.e. some faces of $*_7(2)$ are matched with some faces of $*_8(2)$. 
        \\Here $3\leq t \leq k-1$. 
        \\For the left side, $i_{t-1}=3t-2$, $i_t=3t+2$, when $2\leq j \leq t-1$, $I_j=\{0,1,2\}$. There exists $t\leq l \leq k-1$, such that $I_l=\{0,2\}$. 
        \\For the right side, when $2\leq j \leq t-2$, $I_j=\{0,1,2\}$(we don't consider this situation if $t=3$); $I_{t-1}=\{0,1\}$; when $t\leq j \leq k-1$, $I_j$ is the same as the left side.
        \item ${0 \choose 1}{0,2 \choose 4}{0,1,2 \choose 7}\cdots{0,1,2 \choose 3t-2}{I_t \choose 3t+2}\cdots{I_{m-1} \choose 3m-1}{I_m \choose 3m+3}\cdots{I_{k-1} \choose 3k}\\\leftrightarrow{0 \choose 1}{0,2 \choose 4}{0,1,2 \choose 8} \cdots {0,1,2\choose 3t-4}{0,1 \choose 3t-1}{I_t \choose 3t+2} \cdots {I_{m-1} \choose 3m-1}{I_m \choose 3m+3}\cdots{I_{k-1} \choose 3k}$, i.e. all remaining faces of $*_7(2)$ are matched with some faces of $*_8(2)$.
        \\Here $3\leq t <m \leq k-1$. 
        \\For the left side, $i_t=3t+2$, $i_{t-1}=3t-2$, $i_{m-1}=3m-1$, $i_m=3m+3$, there exists $ t \leq l \leq m-1$, such that $I_l=\{0,2\}$; when $m \leq j \leq k-1$, $I_j \neq \{0\}$. 
        \\For the right side, when $2\leq j \leq t-2$, $I_j=\{0,1,2\}$(we don't consider this situation if $t=3$); $I_{t-1}=\{0,1\}$; when $t\leq j \leq k-1$, $I_j$ is the same as the left side.
    \end{itemize}
    After $\mathcal{M}_7$, the unmatched faces of $*_8(2)$ are:
    \begin{itemize}
        \item ${0\choose 1}{0,2\choose 4}{I_2\choose 8}\cdots {0,2\choose 3m+2}\cdots {I_{k-1}\choose 3k-1}$,
        \\i.e. $2\leq m \leq k-1$;
        \\when $2\leq j \leq m-1$, $I_j=\{0,1,2\}$(we don't consider this situation if m=2);
        \\when $m+1\leq j \leq k-1$, $I_j=\{0,1,2\}$ or $\{0,2\}$(we don't consider this situation $m=k-1$).
        \item ${0\choose 1}{0,2\choose 4}{I_2\choose 8}\cdots{0,2\choose 3m+2}\cdots {I_{t-1}\choose 3t-1}{I_t\choose 3t+3}\cdots {I_{k-1}\choose 3k}$,
        \\i.e. when $2\leq j \leq m-1$, $I_j=\{0,1,2\}$(we don't consider this situation if m=2);$I_m=\{0,2\}$;
        \\when $m+1\leq j \leq t-1$, $I_j=\{0,1,2\}$ or $\{0,2\}$(we don't consider this situation $m=t-1$);
        \\when $t\leq j \leq k-1$, $I_j$ can be any subset of $\{0,1,2\}$ that satisfies the condition of $*_8(2)$.
    \end{itemize}
    For $\mathcal{M}_8$, we have the following matchings:
    \begin{itemize}
        \item ${0 \choose 1}{0,2 \choose 5}{0,1,2 \choose 8} \cdots {0,1,2 \choose 3k-1}
        \leftrightarrow {0 \choose 1}{0,2 \choose 5}{0,1,2 \choose 9}\cdots {0,1 \choose 3k}$.
        \item ${0 \choose 1}{0,2 \choose 5}{0,1,2\choose 8}\cdots{0,1,2 \choose 3t-1}{I_t\choose 3t+3}\cdots {0,2\choose 3m+3}\cdots {I_{k-1}\choose 3k}\\\leftrightarrow {0\choose 1}{0,2 \choose 5}{I_2\choose 9}\cdots {0,1\choose 3t}{I_t\choose 3t+3}\cdots {0,2\choose 3m+3}\cdots {I_{k-1}\choose 3k}$, i.e. some faces of $*_8(7)$ are matched with some faces of $*_9(7)$.
        \\Here $3\leq t\leq m \leq k-1$. 
        \\For the left side, when $2\leq j \leq t-1$, $I_j=\{0,1,2\}$; when $t\leq j \leq m-1$, $I_j=\{0,1\}$ or $\{0,1,2\}$(we don't consider this situation if $m=t+1$); $I_m=\{0,2\}$; when $m+1\leq j \leq k-1$, $I_j=\{0,2\}$ or $\{0,1,2\}$(we don't consider this situation if $m=k-1$). 
        \\For the right side, when $2\leq j \leq t-2$, $I_j=\{0,1,2\}$(we don't consider this situation if $t=3$); $I_{t-1}=\{0,1\}$; when $t\leq j \leq k-1$, $I_j$ is the same as the left side.
        \item ${0 \choose 1}{0,2\choose 5}{0,1,2\choose 8}\cdots {0,1,2\choose 3t-1}{I_t \choose 3t+3}\cdots {I_{k-1}\choose 3k}\leftrightarrow {0 \choose 1}{0,2 \choose 5}{I_2\choose 9}\cdots {0,1\choose 3t}{I_t\choose 3t+3}\cdots {I_{k-1}\choose 3k}$, i.e. all remaining faces of $*_8(7)$ are matched with some faces of $*_9(7)$. 
        \\Here $3\leq t \leq k-1$. 
        \\For the left side, when $2\leq j \leq t-1$, $I_j=\{0,1,2\}$; when $t\leq j \leq k-1$, $I_j=\{0,1,2\}$ or $\{0,1\}$.
        \\For the right side, when $2\leq j \leq t-2$, $I_j=\{0,1,2\}$(we don't consider this situation if $t=3$); $I_{t-1}=\{0,1\}$; when $t\leq j \leq k-1$, $I_j$ is the same as the left side.
    \end{itemize}
    After $\mathcal{M}_8$, the unmatched faces of $*_9(7)$ are:
    \begin{itemize}
        \item ${0\choose 1}{0,2 \choose 5}{I_2 \choose 9}\cdots{I_{k-1}\choose 3k}$,
        \\i.e. when $2\leq j \leq k-1$, $I_j=\{0,1,2\}$ or $\{0,2\}$.
    \end{itemize}
    For $\mathcal{M}_9$, we have the following matchings:
    \begin{itemize}
        \item ${0\choose 1}{0,2\choose 4}{0,1,2 \choose 8}\cdots \leftrightarrow {0\choose 1}{0,2 \choose 4}{0,2 \choose 8}\cdots $, i.e. some faces of $*_8(2)$ are matched with some faces of $*_8(2)$. 
    \end{itemize}
    After $\mathcal{M}_{9}$, the unmatched faces of $*_8(2)$ are:
    \begin{itemize}
        \item ${0\choose 1}{0,2\choose4}{0,2\choose 8}{0,1,2\choose 11}\cdots {0,1,2\choose 3k-1}$.
        \item ${0\choose 1}{0,2 \choose 4}{0,2\choose 8}{I_2 \choose 12}\cdots {I_{k-1}\choose 3k}$,
        \\i.e. when $3\leq j \leq k-1$, $I_j=\{0,1,2\}$ or $\{0,1\}$.
        \item ${0\choose 1}{0,2\choose 4}{0,2\choose 8}{0,1,2\choose 11}\cdots {0,1,2\choose 3t-1}{I_t\choose 3t+3}\cdots {I_{k-1}\choose 3k}$,
        \\i.e. $4\leq t \leq k-1$, $i_t=3t+3$, $i_{t-1}=3t-1$;
        \\when $3\leq j \leq t-1$, $I_j=\{0,1,2\}$;
        \\when $t\leq j \leq k-1$, $I_j=\{0,1\}$ or $\{0,1,2\}$.
    \end{itemize}
    For $\mathcal{M}_{10}$, we have the following matchings:
    \begin{itemize}
        \item ${0\choose 1}{0,2 \choose 5}{0,1,2\choose 9}\cdots {I_{k-1}\choose 3k}\leftrightarrow {0\choose 1}{0,2 \choose 5}{0,2\choose 9}\cdots {I_{k-1}\choose 3k}$, i.e. some faces of $*_9(7)$ are matched with all remaining faces of $*_9(7)$.
    \end{itemize}
    For $\mathcal{M}_{11}$, we have the following matchings:
    \begin{itemize}
        \item ${0 \choose 1}{0,2\choose 4}{0,2 \choose 8}{0,1,2 \choose 11}\cdots {0,1,2 \choose 3k-1}\leftrightarrow {0 \choose 1}{0,2 \choose 4}{0,2\choose 8}{0,1,2\choose 12}\cdots {0,1 \choose 3k}$.
        \item ${0 \choose 1}{0,2\choose 4}{0,2 \choose 8}{0,1,2 \choose 11}\cdots {0,1,2\choose 3t-1}{I_t\choose 3t+3}\cdots{I_{k-1} \choose 3k}\leftrightarrow {0 \choose 1}{0,2 \choose 4}{0,2\choose 8}{0,1,2\choose 12}\cdots {0,1\choose 3t}{I_t\choose 3t+3}\cdots {I_{k-1}\choose 3k}$. 
        \\Here $4 \leq t \leq k-1$. 
        \\For the left side, $i_t=3t+3$, $i_{t-1}=3t-1$, when $3\leq j \leq t-1$, $I_j=\{0,1,2\}$; when $t \leq j \leq k-1$, $I_j=\{0,1\}$ or $\{0,1,2\}$. 
        \\For the right side, when $3\leq j \leq t-2$, $I_j=\{0,1,2\}$(we don't consider this situation if $t=4$); $I_{t-1}=\{0,1\}$; when $t \leq j \leq k-1$, $I_j$ is the same as the left side.
    \end{itemize}
    And the only unmatched face is $[n]\setminus {0\choose 1}{0,2\choose 4}{0,2\choose 8}{0,1,2 \choose 12}\cdots {0,1,2 \choose 3k}=\{2,3,5,7,9,11\}$. By discrete Morse theory, we have $\Delta_{k}^t(W_{3k+2})\simeq \mathbb{S}^{5}$.
\end{proof}

\section{Conclusions and Further Directions}\label{sec5:conclu}
In \cite{Bayer_2024_02} and \cite{Bayer_2024}, the authors  determined the homotopy type of the complex $\Delta_k^t(C_n)$ and $\Delta_k(C_n)$ for all $k$. They also proved that the complex $\Delta_2^t(W_n)$ ($n\ge7$) is homotopy equivalent to a single sphere of dimension $n-4$. Later in \cite{3-cut_complexes}, the authors proved that the complex $\Delta_3(W_n)$ of the squared cycle graphs is shellable for all
$n \ge 9$, by constructing a shelling order. In this paper, we proved that the complex $\Delta_3^t(W_n)$ ($n\ge 9)$ is homotopy equivalent to a single sphere of dimension $n-6$. Moreover, we construct a new approach to describe the faces of $\Delta_k^t(W_n)$, and solve the condition when $n=3k+1$ and $n=3k+2$. We believe that this method can be generalized
to prove the homotopy type of $\Delta_k^t(W_{3k+i})$ for all $k\ge 3$ and $1\le i\le k-1$.

Furthermore, by using SageMath, we calculated the homology groups of $\Delta_k^t(C_n^p)$, namely $k$-total cut complexes of powers of cycles (with coefficient $\mathbb{Z}$), and conjectured general cases for total cut complexes of powers of cycles.  We have data for some small values of $k$, $p$ and $n$, which is illustrated in Tables \ref{tab:total-2-cut}-\ref{tab:total-4-cut}. Missing entries in these tables indicate that the corresponding complex $\Delta_k^t(C_n^p)$ is void. An entry of the form $i : \mathbb{Z}^\beta$ denotes that the $i$-th homology group is $\mathbb{Z}^\beta$.
Through these tables, we propose the fowlling conjectures.
\begin{table}[ht]
    \tiny
    \centering
    \begin{tabular}{|c|c|c|c|c|c|c|c|c|c|c|c|}
     \hline
    \diagbox[linewidth=0.2pt, width=.6cm, height=.6cm]{$p$}{$n$}  &8 &9 &10 &11 &12 &13 &14 &15 &16 \\
    \hline
      3 &$2: \mathbb{Z}^{}$ &$4: \mathbb{Z}^{2}$& $6: \mathbb{Z}^{}$& $7: \mathbb{Z}^{}$& $8: \mathbb{Z}^{}$& $9: \mathbb{Z}^{}$ & $10: \mathbb{Z}^{}$ & $11: \mathbb{Z}^{}$ & $12: \mathbb{Z}^{}$\\
    \hline
      4 & && $3: \mathbb{Z}^{}$& $5: \mathbb{Z}^{}$& $7: \mathbb{Z}^{3}$& $9: \mathbb{Z}^{}$ & $10: \mathbb{Z}^{}$ & $11: \mathbb{Z}^{}$ & $12: \mathbb{Z}^{}$\\
    \hline
      5 & && & & $4: \mathbb{Z}^{}$& $7: \mathbb{Z}^{}$ & $8: \mathbb{Z}^{}$ & $10: \mathbb{Z}^{4}$ & $12: \mathbb{Z}^{}$\\
    \hline
      6  & && & &  & & $5: \mathbb{Z}^{}$ & $8: \mathbb{Z}^{2}$ & $10: \mathbb{Z}^{}$\\
    \hline
    \end{tabular}
    \caption{ Homology  of total $2$-cut complexes $\Delta_{2}^t(C_n^p)$}
    \label{tab:total-2-cut}
\end{table}
\begin{conj}\label{conj:total-2-cut}
    For $p\ge3$ and $n\ge 3p+1$, the complex $\Delta_2^t(C_n^p)$ is homotopy equivalent to $\mathbb{S}^{n-4}$.
\end{conj}
Notice in \cite{3-cut_complexes}, the authors also conjectured the homology groups of $2$-cut complexes $\Delta_2(C_n^p)$ and surprisingly these two hypotheses are consistent. This is due to the definition of these two complexes is consistent when $k=2$.

For general integers $n$, $k$ and $p$, we are more concerned about the performance of $\Delta_k^t(C_n^p)$ when $n$ is large enough, as the following tables show.
\begin{table}[ht]
    \centering\tiny
    \begin{tabular}{|c|c|c|c|c|c|}
    \hline
    \diagbox[linewidth=0.3pt, width=.85cm, height=.9cm]{$n$}{$p$} &2 &3 &4 &5 & 6 \\
    \hline
        11 &$5:\mathbb{Z}^{}$&&&&\\
    \hline
        16  &$10:\mathbb{Z}^{}$&$10:\mathbb{Z}^{}$&&&\\
    \hline
        21  &$15:\mathbb{Z}^{}$&$15:\mathbb{Z}^{}$&$15:\mathbb{Z}^{}$&&\\
    \hline
        26 &$20:\mathbb{Z}^{}$&$20:\mathbb{Z}^{}$&$20:\mathbb{Z}^{}$&$20:\mathbb{Z}^{}$& \\
    \hline
        31  &$25:\mathbb{Z}^{}$ &$25:\mathbb{Z}^{}$ &$25:\mathbb{Z}^{}$ &$25:\mathbb{Z}^{}$ &$25:\mathbb{Z}^{}$\\
    \hline
       32  &$26:\mathbb{Z}^{}$ &$26:\mathbb{Z}^{}$ &$26:\mathbb{Z}^{}$ &$26:\mathbb{Z}^{}$ &$26:\mathbb{Z}^{}$\\
    \hline
        33
        &$27:\mathbb{Z}^{}$ &$27:\mathbb{Z}^{}$ &$27:\mathbb{Z}^{}$ &$27:\mathbb{Z}^{}$ &$27:\mathbb{Z}^{}$\\
    
    \hline
    \end{tabular}
    \caption{Homology of total $3$-cut complexes $\Delta_{3}^t(C_n^p)$} 
    \label{tab:total-3-cut}
\end{table}
\begin{table}[ht]
    \centering\tiny
    \begin{tabular}{|c|c|c|c|c|c|}
    \hline
    \diagbox[linewidth=0.3pt, width=.85cm, height=.9cm]{$n$}{$p$} &2 &3 &4 &5 & 6 \\
    \hline
        15 &$7:\mathbb{Z}^{}$&&&&\\
    \hline
        16  &$8:\mathbb{Z}^{}$&$2:\mathbb{Z}^{}$&&&\\
    \hline
        22  &$14:\mathbb{Z}^{}$&$14:\mathbb{Z}^{}$&&&\\
    \hline
        29 &$21:\mathbb{Z}^{}$&$21:\mathbb{Z}^{}$&$21:\mathbb{Z}^{}$&& \\
    \hline
        36  &$28:\mathbb{Z}^{}$ &$28:\mathbb{Z}^{}$ &$28:\mathbb{Z}^{}$ &$28:\mathbb{Z}^{}$ &\\
    \hline
       43  &$35:\mathbb{Z}^{}$ &$35:\mathbb{Z}^{}$ &$35:\mathbb{Z}^{}$ &$35:\mathbb{Z}^{}$ &$35:\mathbb{Z}^{}$\\
    \hline
        44
        &$36:\mathbb{Z}^{}$ &$36:\mathbb{Z}^{}$ &$36:\mathbb{Z}^{}$ &$36:\mathbb{Z}^{}$ &$36:\mathbb{Z}^{}$\\
    
    \hline
    \end{tabular}
    \caption{Homology of total $4$-cut complexes $\Delta_{4}^t(C_n^p)$} 
    \label{tab:total-4-cut}
\end{table}

Based on our calculation given in \Cref{tab:total-3-cut} and \Cref{tab:total-4-cut}, we make the following conjecture for $n\ge (2k-1)p+1$.
\begin{conj}\label{conj:all}
    For $p\ge 3$ and $n\ge (2k-1)p+1 $, the complex $\Delta_k^t(C_n^p)$ is homotopy equivalent to a single sphere of dimension $n-2k$.
\end{conj}

\section{Acknowledgments}

The authors thank the organizers of the 2025 PKU Algebraic Combinatorics Experience, where this work originated. We also thank Lei Xue, our mentor who provided us a lot of help.

We are also grateful to the anonymous referees for their careful reading of the paper.

\newpage

\renewcommand*{\bibfont}{\footnotesize}

    \bibliographystyle{apalike}
\end{document}